\documentclass[11pt,a4paper]{article}

\usepackage[utf8]{inputenc}
\usepackage[T1]{fontenc}
\usepackage{lmodern}
\usepackage{amsmath,amssymb,amsthm,mathtools}
\usepackage{geometry}
\usepackage{enumitem}
\usepackage{hyperref}

\hypersetup{
	hidelinks,
	pdftitle={Ambient Hardy--Littlewood Maximal Functions on Weighted Musielak--Orlicz Spaces over Domains},
	pdfauthor={Tan Duc Do}
}
\usepackage{microtype}
\usepackage{graphicx}
\newcommand{\fint}{\mathop{\mathchoice
		{\vcenter{\hbox{\scalebox{1.2}{$\displaystyle\int$}}}}
		{\vcenter{\hbox{\scalebox{1.0}{$\textstyle\int$}}}}
		{\vcenter{\hbox{\scalebox{0.8}{$\scriptstyle\int$}}}}
		{\vcenter{\hbox{\scalebox{0.6}{$\scriptscriptstyle\int$}}}}}\displaylimits\kern-1.0em -}

\theoremstyle{plain}
\newtheorem{theorem}{Theorem}[section]
\newtheorem{lemma}[theorem]{Lemma}
\newtheorem{proposition}[theorem]{Proposition}
\newtheorem{corollary}[theorem]{Corollary}

\theoremstyle{definition}
\newtheorem{definition}[theorem]{Definition}
\newtheorem{example}[theorem]{Example}

\theoremstyle{remark}
\newtheorem{remark}[theorem]{Remark}

\title{Ambient Hardy--Littlewood Maximal Functions on Weighted Musielak--Orlicz Spaces over Domains}
\author{
	Tan Duc Do \\
	\small Faculty of Applied Sciences, Ho Chi Minh City University of Industry and Trade \\
	\small 140 Le Trong Tan Street, Tay Thanh Ward, Ho Chi Minh City, Vietnam \\
	\small \texttt{tanducdo.math@gmail.com}
}
\date{}

\begin{document}
	\maketitle
	
	\begin{abstract}
		We study the ambient-domain Hardy--Littlewood maximal operator 
		\[
		\mathcal M_\Omega f(x)
		:=
		\sup_{B\ni x}\frac1{|B|}\int_{B\cap\Omega}|f(y)|\,dy,
		\qquad x\in\Omega
		\]
		on weighted Musielak--Orlicz spaces over a general open set \(\Omega\subset\mathbb R^n\), where the supremum is taken over all Euclidean balls \(B\subset\mathbb R^n\). For a Musielak--Orlicz function \(\varphi\), we use the pointwise lower Matuszewska--Orlicz index \(p_\varphi(\cdot)\) and the lower-index normalization
		\[
		\psi_\varphi(x,t)=\varphi(x,t)^{1/p_\varphi(x)}.
		\]
		This factorizes the modular as a weighted variable-exponent modular applied to \(\psi_\varphi(x,|f|)\). Under endpoint lower growth, normalized weighted generalized Orlicz \((A0)\)--\((A2)\) assumptions and an admissible whole-space extension hypothesis for the weight at the lower-index exponent, we prove the boundedness of
		\[
		\mathcal M_\Omega:L^{\varphi(\cdot)}_\omega(\Omega)\to L^{\varphi(\cdot)}_\omega(\Omega).
		\]
		For the converse direction we use the natural K\"othe-associate ambient ball condition \(A_\varphi(\Omega)\). Under the local characteristic-function hypothesis, boundedness of \(\mathcal M_\Omega\) implies \(\omega\in A_\varphi(\Omega)\). 
		On the whole space \(\mathbb R^n\), this framework provides a weighted characterization conditional on an associate-to-lower-index product reduction. In the present paper this reduction is verified for uniformly lower-index-power-equivalent models; it remains open for genuinely two-phase growth such as \(t^p+a(x)t^q\).
		As an application, we prove density of \(C_c^\infty(\mathbb R^n)\) in weighted Musielak--Orlicz--Sobolev spaces.
	\end{abstract}
	
	

	\section{Introduction}
	
	The Hardy--Littlewood maximal operator is one of the central objects of real-variable
	harmonic analysis. For \(f\in L^1_{\mathrm{loc}}(\mathbb R^n)\), it is defined by
	\[
	Mf(x):=\sup_{B\ni x}\fint_B |f(y)|\,dy,
	\]
	where the supremum is taken over all Euclidean balls \(B\subset\mathbb R^n\) containing
	\(x\). The classical theorem asserts that
	\[
	M:L^p(\mathbb R^n)\longrightarrow L^p(\mathbb R^n)
	\]
	is bounded for \(1<p<\infty\). In the weighted setting, Muckenhoupt's theorem gives the
	sharp condition
	\[
	M:L^p_\omega(\mathbb R^n)\longrightarrow L^p_\omega(\mathbb R^n)
	\quad\Longleftrightarrow\quad
	\omega\in A_p .
	\]
	We refer to \cite{Stein1970,Muckenhoupt1972} for the classical theory.
	
	Several extensions of this theorem are relevant here. In variable exponent spaces, the
	constant exponent \(p\) is replaced by a measurable exponent \(p(\cdot)\) and the sharp
	weighted condition is the variable Muckenhoupt class \(A_{p(\cdot)}\); see
	\cite{CruzUribeDieningHasto2011,DieningHarjulehtoHastoRuzicka2011}. In weighted scalar
	Orlicz spaces one must distinguish modular inequalities from boundedness with respect to the
	Luxemburg norm. Classical results such as
	\cite{BloomKerman1994,KokilashviliKrbec1991} primarily concern the stronger modular
	problem and do not in general identify the weight class for the norm inequality. For the
	norm problem, the universal necessary condition is naturally expressed through the
	K\"othe associate and the complementary Young function; see also
	\cite{KrasnoselskiiRutickii1961}. In generalized
	Orlicz and Musielak--Orlicz spaces, where the growth depends on both \(x\) and \(t\),
	boundedness requires structural conditions of \((A0)\), \((A1)\), \((A2)\)-type, together
	with an appropriate lower-growth condition. The weighted result of Hietanen
	\cite{Hietanen2025} is a useful model: one assumes a classical Muckenhoupt condition on the
	weight, standard generalized Orlicz conditions and an almost-increasing condition of the
	form \(t\mapsto \varphi(x,t)/t^p\). Our notation below follows this generalized Orlicz
	convention. Recent applications of variable-exponent and Musielak--Orlicz techniques to
	interpolation inequalities and Schr\"odinger estimates can be found in
	\cite{DoThanhTrong2022,DoTruongTrong2026,TruongTrongDoLam2024}.
	
	Relative to Hietanen \cite{Hietanen2025}, we do not claim the ambient-domain formulation itself as new. Hietanen proves a sufficiency theorem for this operator by imposing a fixed exponent \(r>1\), the classical condition \(\omega\in A_r\), a direct \((aInc)_r\) condition on \(\varphi\), and weighted generalized Orlicz structural assumptions on \(\varphi\). The present paper uses a different mechanism and adds different conclusions. First, the pointwise lower index \(p_\varphi(\cdot)\) is allowed to vary, and the normalization \(\psi_\varphi=\varphi^{1/p_\varphi(\cdot)}\) transfers the maximal estimate to the weighted variable-exponent class at \(p_\varphi(\cdot)\). For example, when \(\varphi(x,t)=t^{p(x)}\), the normalized function is simply \(t\), and the resulting whole-space weight condition is the natural variable class \(A_{p(\cdot)}\), rather than a classical \(A_r\)-condition at a fixed lower exponent. Second, we prove an independent necessity theorem in terms of the K\"othe-associate ambient ball class \(A_\varphi\). Third, we isolate the additional product reduction needed to pass from this necessary class back to the lower-index class, and we derive a density theorem for weighted Musielak--Orlicz--Sobolev spaces. Thus the novelty lies in the lower-index/variable-exponent transfer, the associate-space converse, the conditional characterization mechanism and the density application, not in the ambient operator alone. The two sufficiency packages are not claimed to dominate one another in complete generality.

	The purpose of this paper is to develop this lower-index and associate-space theory for the
	Hardy--Littlewood maximal operator on weighted Musielak--Orlicz spaces over domains. Let
	\(\Omega\subset\mathbb R^n\) be open. We use the zero-extension domain maximal operator
	\begin{equation*}
		\mathcal M_\Omega f(x):=
		\sup_{B\ni x}
		\frac1{|B|}\int_{B\cap\Omega}|f(y)|\,dy,
		\qquad x\in\Omega .
	\end{equation*}
	Equivalently, if \(\widetilde f=f\chi_\Omega\) denotes the zero extension of \(f\) to
	\(\mathbb R^n\), then
	\[
	\mathcal M_\Omega f=(M\widetilde f)|_\Omega .
	\]
	This is the ambient-domain formulation used in the weighted generalized Orlicz theory of
	Hietanen. Its advantage for the present lower-index method is that the variable-exponent
	maximal step is supplied by the usual ambient weighted variable-exponent theorem through the
	whole-space extension hypothesis.
	
	Let \(\varphi:\Omega\times[0,\infty)\to[0,\infty]\) be a weak Musielak--Orlicz function.
	For each \(0<\lambda<1\), set
	\[
	g_\varphi(x,\lambda):=
	\sup_{\substack{t>0\\0<\varphi(x,t)<\infty}}
	\frac{\varphi(x,\lambda t)}{\varphi(x,t)} .
	\]
	We consider this quantity only at points \(x\) for which the admissible set in the
	supremum is nonempty. Whenever the following limit exists, we define the pointwise lower
	Matuszewska--Orlicz index by
	\begin{equation*}
		p_\varphi(x):=i_\varphi(x):=
		\lim_{\lambda\to0+}
		\frac{\log g_\varphi(x,\lambda)}{\log\lambda}.
	\end{equation*}
	Throughout the paper we use the lower-index normalization
	\begin{equation*}
		\psi_\varphi(x,t):=\varphi(x,t)^{1/p_\varphi(x)} .
	\end{equation*}
	Thus
	\[
	\varphi(x,t)=\psi_\varphi(x,t)^{p_\varphi(x)}.
	\]
	Consequently,
	\[
	\int_\Omega \varphi(x,|f(x)|)\omega(x)\,dx
	=
	\int_\Omega
	\psi_\varphi(x,|f(x)|)^{p_\varphi(x)}
	\omega(x)\,dx .
	\]
	This identity is the lower-index factorization behind the sufficiency theorem. It reduces
	the Musielak--Orlicz maximal problem to a weighted variable exponent maximal estimate at
	the exponent \(p_\varphi(\cdot)\).
	
	The converse direction is treated differently. In scalar Orlicz theory, necessity is not
	proved through an auxiliary lower-index testing condition. It is expressed through the
	K\"othe associate space and the complementary Young function. We adopt this viewpoint in
	the Musielak--Orlicz setting. If
	\[
	X:=L^{\varphi(\cdot)}_\omega(\Omega),
	\]
	and \(X'\) denotes the K\"othe associate of \(X\) with respect to Lebesgue measure, then
	under the local characteristic-function hypothesis, boundedness of \(\mathcal M_\Omega\) forces the ambient associate ball condition
	\[
	\sup_{B\subset\mathbb R^n}
	\frac{\|\chi_{B\cap\Omega}\|_X\|\chi_{B\cap\Omega}\|_{X'}}{|B|}<\infty .
	\]
	We denote this condition by
	\[
	\omega\in A_\varphi(\Omega).
	\]
	Thus, under this local characteristic-function hypothesis, \(A_\varphi(\Omega)\) is the natural necessary class for
	\(\mathcal M_\Omega\). The lower-index Muckenhoupt condition
	\[
	\omega\in A_{p_\varphi(\cdot)}(\Omega)
	\]
	is the robust testing class behind the sufficiency estimates; in the domain theorem it is
	supplied by the extension hypothesis \((E_\Omega)_\omega\). These two classes coincide in
	pure-power and standard variable-exponent situations. For a general weighted scalar Orlicz
	space, however, the classical modular theory does not identify the norm-level associate
	condition \(A_\varphi(\Omega)\) with the lower-index class. We therefore isolate an
	associate-to-lower-index product reduction as an additional hypothesis which identifies
	\(A_\varphi(\Omega)\) with \(A_{p_\varphi(\cdot)}(\Omega)\) only in those cases where
	such an identification has been established independently.

	Two limitations should be emphasized. First, the genuine double-phase model
	\[
	\varphi(x,t)=t^p+a(x)t^q,\qquad 1<p<q,
	\]
	is covered by the sufficiency theorem and by the associate-space necessity theorem, but the product reduction is not established; hence the conditional characterization does not close for this model. All reduction verifications given below ultimately use uniform equivalence to the lower-index power \(t^{p_\varphi(x)}\). Finding a non-power-equivalent class for which the reduction holds is left open.

	Second, the domain sufficiency theorem is deliberately formulated under the whole-space extension hypothesis \((E_\Omega)_\omega\). This hypothesis allows us to apply the whole-space weighted variable-exponent maximal theorem to the zero extension, but it is not asserted to be necessary or equivalent to the induced condition \(\omega\in A_{p_\varphi(\cdot)}(\Omega)\). In fixed-exponent settings it can be checked by classical extension results under suitable strengthened assumptions. For genuinely variable exponents on arbitrary, possibly non-smooth domains, no general extension criterion is presently used here; consequently \((E_\Omega)_\omega\) is a real restriction and the principal bottleneck for domain applications.
	
	\subsection{Organization of the paper}

	Section~\ref{sec:framework-main} gives the precise ambient weight classes, structural and extension hypotheses, and statements of the main results. Section~\ref{sec:preliminaries} collects the auxiliary notation and function-space facts. Section~\ref{sec:domain-ve-estimates} records the weighted variable-exponent estimates needed later and explains the ambient maximal principle. Section~\ref{sec:local-mo-estimates} proves the normalized local Musielak--Orlicz estimates. Section~\ref{sec:proof-main-theorems} proves the boundedness and necessity results and records the whole-space consequence. Section~\ref{sec:density} proves the density theorem, and Section~\ref{sec:examples} verifies the assumptions in model classes.

	\section{Framework and main results}\label{sec:framework-main}

	\subsection{Ambient lower-index Muckenhoupt condition}
	
	Let \(p\in P(\Omega)\), that is,
	\[
	1<p^-_\Omega\le p^+_\Omega<\infty.
	\]
	Throughout the paper, a weight on \(\Omega\) means a measurable function \(\omega\) satisfying
	\[
	0<\omega(x)<\infty \quad\text{for a.e. }x\in\Omega,
	\qquad
	\omega\chi_\Omega\in L^1_{\mathrm{loc}}(\mathbb R^n).
	\]
	Equivalently,
	\[
	\int_{B\cap\Omega}\omega(x)\,dx<\infty
	\]
	for every bounded Euclidean ball \(B\subset\mathbb R^n\). We also write
	\(\omega\in L^1_{\mathrm{loc}}(\overline\Omega)\) for this condition.
	
	\begin{definition} \label{def:ambient-ap}
		Let \(p\in P(\Omega)\). We say that
		\[
		\omega\in A_{p(\cdot)}(\Omega)
		\]
		if
		\[
		[\omega]_{A_{p(\cdot)}(\Omega)}
		:=
		\sup_{B\subset\mathbb R^n}
		\frac{
			\|\chi_{B\cap\Omega}\|_{L^{p(\cdot)}_\omega(\Omega)}
			\|\omega^{-1/p(\cdot)}\chi_{B\cap\Omega}\|_{L^{p'(\cdot)}(\Omega)}
		}{|B|}
		<\infty .
		\]
		The supremum is taken over all Euclidean balls \(B\subset\mathbb R^n\) with
		\(B\cap\Omega\ne\emptyset\).
	\end{definition}
	
	For constant \(p\), Definition~\ref{def:ambient-ap} gives the ambient-domain form
	\[
	\sup_{B\subset\mathbb R^n}
	\left(\frac1{|B|}\int_{B\cap\Omega}\omega\,dx\right)
	\left(\frac1{|B|}\int_{B\cap\Omega}\omega^{-\frac1{p-1}}\,dx\right)^{p-1}<\infty .
	\]
	This is the relative ambient-domain \(A_p\)-condition associated with
	\(\mathcal M_\Omega\). If \(\omega\) is the restriction to \(\Omega\) of a
	whole-space \(A_p\)-weight, then this condition follows immediately.
	
	\subsection{Associate Musielak--Orlicz condition}
	
	Let
	\[
	X:=L^{\varphi(\cdot)}_\omega(\Omega).
	\]
	The K\"othe associate \(X'\) is taken with respect to Lebesgue measure:
	\[
	\|g\|_{X'}
	:=
	\sup_{\|f\|_X\le1}
	\int_\Omega |f(x)g(x)|\,dx .
	\]
	
	\begin{definition}
		We say that
		\[
		\omega\in A_\varphi(\Omega)
		\]
		if
		\[
		[\omega]_{A_\varphi(\Omega)}
		:=
		\sup_{B\subset\mathbb R^n}
		\frac{
			\|\chi_{B\cap\Omega}\|_{L^{\varphi(\cdot)}_\omega(\Omega)}
			\|\chi_{B\cap\Omega}\|_{(L^{\varphi(\cdot)}_\omega(\Omega))'}
		}{|B|}
		<\infty .
		\]
	\end{definition}
	
	When \(\varphi\) has an equivalent convex representative and the standard
	Musielak--Orlicz duality theorem applies, the K\"othe associate can be
	represented through the pointwise complementary function \(\varphi^*(x,\cdot)\).
	In that case the second factor can be written as
	\[
	\|\chi_{B\cap\Omega}\|_{(L^{\varphi(\cdot)}_\omega(\Omega))'}
	\simeq
	\|\omega^{-1}\chi_{B\cap\Omega}\|_{L^{\varphi^*(\cdot)}_\omega(\Omega)} .
	\]
	Thus \(A_\varphi(\Omega)\) is the Musielak--Orlicz analogue of the
	complementary-function condition in scalar Orlicz theory, but its definition
	itself is made through the K\"othe associate and does not require convexity.
	
	\begin{definition}
		\label{def:associate-lower-index-reduction}
		Assume that \(p_\varphi\) is defined. We say that
		\(L^{\varphi(\cdot)}_\omega(\Omega)\) satisfies the \emph{associate-to-lower-index product reduction} if there exists \(C\ge1\) such that, for every ball \(B\subset\mathbb R^n\),
		\[
		\|\omega^{1/p_\varphi(\cdot)}\chi_{B\cap\Omega}\|_{L^{p_\varphi(\cdot)}(\Omega)}
		\,
		\|\omega^{-1/p_\varphi(\cdot)}\chi_{B\cap\Omega}\|_{L^{p_\varphi'(\cdot)}(\Omega)}
		\le
		C
		\|\chi_{B\cap\Omega}\|_{L^{\varphi(\cdot)}_\omega(\Omega)}
		\,
		\|\chi_{B\cap\Omega}\|_{(L^{\varphi(\cdot)}_\omega(\Omega))'} .
		\]
	\end{definition}
	
	The product formulation compares the one-ball associate testing quantity with its
	lower-index counterpart. It is automatic in the pure-power model and follows from standard
	associate duality in the variable-exponent model. It is not asserted here for a general
	weighted scalar Orlicz space. In particular, modular inequalities from classical weighted
	Orlicz theory do not by themselves establish this norm-level product reduction. Every explicit verification in this paper is obtained through uniform equivalence of \(\varphi(x,t)\) to the lower-index power \(t^{p_\varphi(x)}\); no non-power-equivalent example is presently claimed. Whether the reduction can hold in a genuinely non-power-equivalent regime is an open question.
	
	\subsection{Structural assumptions}
	
	We now state the structural assumptions used in the sufficiency theorem. The notation is
	chosen to follow the generalized Orlicz convention used in
	\cite{HarjulehtoHasto2019,Hietanen2025}: normalization \((A0)\), local inverse
	comparability \((A1)\), bounded-level inverse comparability \((A2)\), lower growth
	\((aInc)\) and upper growth \((aDec)\). For \((A2)\), we use the weighted analogue of the corrected
	shifted-level inverse formulation introduced in
	\cite[Definition~2.1]{HarjulehtoHastoSlabuszewski2024}. The conditions \((A0)\),
	\((A1)\) and \((A2)\) are stated for a general weak Musielak--Orlicz function. In
	the standing hypotheses below they are applied to the normalized function
	\(\psi_\varphi\), with \((A1)\) formulated on the ambient domain ball basis.

		A function \(\Phi:\Omega\times[0,\infty)\to[0,\infty]\) is called a weak
	\(\Phi\)-function if \(\Phi(\cdot,|f|)\) is measurable for every measurable
	\(f:\Omega\to\mathbb R\), and, for a.e. \(x\in\Omega\),
	\(t\mapsto\Phi(x,t)\) is increasing,
	\[
	\Phi(x,0)
	=
	\lim_{t\to0+}\Phi(x,t)
	=0,
	\qquad
	\lim_{t\to\infty}\Phi(x,t)=\infty,
	\]
	and \(t\mapsto\Phi(x,t)/t\) is almost increasing on \((0,\infty)\), uniformly in
	\(x\). The class of such functions is denoted by \(\Phi_w(\Omega)\); see
	\cite[Definition~1.4]{HarjulehtoHastoSlabuszewski2024}.
	
	The generalized inverse of \(\Phi(x,\cdot)\) is
	\[
	\Phi^{-1}(x,\tau)
	:=
	\inf\{t\ge0:\Phi(x,t)\ge \tau\},
	\qquad \tau\ge0.
	\]
	This is precisely the left-inverse convention of
	\cite[Definition~2.3.1]{HarjulehtoHasto2019} and
	\cite[Definition~1.4]{HarjulehtoHastoSlabuszewski2024}. All inverse conditions in
	the paper are understood in this sense. No strict monotonicity of
	\(t\mapsto\Phi(x,t)\) is imposed.

	\begin{definition}[\((A0)^\Omega\) condition]
		Let \(\Phi\in\Phi_w(\Omega)\). We say that \(\Phi\) satisfies \((A0)^\Omega\)
		if there exists \(\beta_0\in(0,1]\) such that, for a.e. \(x\in\Omega\),
		\[
		\beta_0
		\le
		\Phi^{-1}(x,1)
		\le
		\beta_0^{-1}.
		\]
		Equivalently, after replacing \(\beta_0\) by a smaller structural
		constant if necessary,
		\[
		\Phi(x,\beta_0)\le1
		\le
		\Phi(x,\beta_0^{-1})
		\]
		for a.e. \(x\in\Omega\).
	\end{definition}
	
	\begin{definition}[\((A1)^\Omega_{\omega,p}\) condition]

		Let \(\Phi\in\Phi_w(\Omega)\) and let \(p\in P(\Omega)\). We say that \(\Phi\)
		satisfies \((A1)^\Omega_{\omega,p}\) if there exists \(\beta_1\in(0,1]\) such that,
		for every ball \(B\subset\mathbb R^n\) satisfying
		\[
		\|\chi_{B\cap\Omega}\|_{L^{p(\cdot)}_\omega(\Omega)}\le1,
		\]
		for a.e. \(x,y\in B\cap\Omega\), and for every
		\[
		1\le s\le
		\|\chi_{B\cap\Omega}\|_{L^{p(\cdot)}_\omega(\Omega)}^{-1},
		\]
		one has
		\[
		\beta_1\Phi^{-1}(y,s)
		\le
		\Phi^{-1}(x,s).
		\]
	\end{definition}
	
	As a direct consequence, there exist \(\gamma_1\in(0,1]\) and \(C\ge1\),
	depending only on the weak \(\Phi\)-data and the constant in
	\((A1)^\Omega_{\omega,p}\), such that whenever
	\[
	1\le \Phi(y,t)
	\le
	\|\chi_{B\cap\Omega}\|_{L^{p(\cdot)}_\omega(\Omega)}^{-1},
	\]
	one has
	\[
	\Phi(x,\gamma_1 t)
	\le
	C\Phi(y,t).
	\]
	
	\begin{remark}
		In the direct consequence of \((A1)^\Omega_{\omega,p}\), the constant \(C\) can be absorbed after decreasing \(\gamma_1\), in the sense that the estimate
		\[
		\Phi(x,\gamma_1 t)\le C\Phi(y,t)
		\]
		may be replaced by the same estimate with constant \(1\) on the right-hand
		side. Indeed, since \(\Phi\in\Phi_w(\Omega)\), there is \(L\ge1\) such that
		\[
		\Phi(x,\theta r)\le L\theta\,\Phi(x,r),
		\qquad 0<\theta\le1,\ r>0 .
		\]
		Choosing \(\theta=(LC)^{-1}\) and replacing \(\gamma_1\) by
		\(\widetilde\gamma_1=\theta\gamma_1\), we obtain
		\[
		\begin{aligned}
			\Phi(x,\widetilde\gamma_1 t)
			&=
			\Phi(x,\theta\gamma_1 t)  
			\le
			L\theta\,\Phi(x,\gamma_1 t) 
			\le
			L\theta C\,\Phi(y,t) 
			=
			\Phi(y,t).
		\end{aligned}
		\]
		Thus, after relabelling \(\widetilde\gamma_1\) again as \(\gamma_1\), the
		constant \(C\) is absorbed.
	\end{remark}
	
	\begin{definition}[\((A2)^\Omega_\omega\) condition]
		\label{def:A2-weighted}
		Let \(\Phi\in\Phi_w(\Omega)\). We say that \(\Phi\) satisfies \((A2)^\Omega_\omega\)
		if, for every \(s>0\), there exist \(\beta_s\in(0,1]\) and a function
		\[
		h_s\in L^1_\omega(\Omega)\cap L^\infty(\Omega),
		\qquad h_s\ge0,
		\]
		such that, for a.e. \(x,y\in\Omega\) and every \(\tau\in[0,s]\),
		\[
		\beta_s\Phi^{-1}(y,\tau)
		\le
		\Phi^{-1}\!\left(x,\tau+h_s(x)+h_s(y)\right).
		\]
	\end{definition}
	
	As a direct consequence, for every \(s>0\) there exist
	\(\gamma_s\in(0,1]\), \(C_s\ge1\) and
	\(h_s\in L^1_\omega(\Omega)\cap L^\infty(\Omega)\), \(h_s\ge0\), such that
	\[
	\Phi(x,\gamma_s t)
	\le
	C_s\{\Phi(y,t)+h_s(x)+h_s(y)\}
	\]
	whenever
	\[
	0\le \Phi(y,t)\le s.
	\]

	\begin{definition}[Almost-increasing lower growth \((aInc)_{p(\cdot)}\)]

		Let \(\Phi\in\Phi_w(\Omega)\) and let \(p:\Omega\to(0,\infty)\) be measurable.
		We say that \(\Phi\) satisfies \((aInc)_{p(\cdot)}\) if there exists \(L\ge1\)
		such that, for a.e. \(x\in\Omega\),
		\[
		\frac{\Phi(x,s)}{s^{p(x)}}
		\le
		L
		\frac{\Phi(x,t)}{t^{p(x)}}
		\]
		for all \(0<s<t\).
	\end{definition}
	
	Applied to \(\Phi=\varphi\) and \(p=p_\varphi\), this is the variable lower-index
	analogue of the almost-increasing assumption
	\[
	t\mapsto \frac{\varphi(x,t)}{t^p}
	\]
	used in the fixed-index weighted generalized Orlicz theorem.
	
	\begin{definition}[Almost-decreasing upper growth \((aDec)_q\)]

		Let \(\Phi\in\Phi_w(\Omega)\). We say that \(\Phi\) satisfies \((aDec)_q\) if
		there exist \(q\in[1,\infty)\) and \(C\ge1\) such that, for a.e.
		\(x\in\Omega\),
		\[
		\frac{\Phi(x,s)}{s^q}
		\ge
		C^{-1}
		\frac{\Phi(x,t)}{t^q}
		\]
		for all \(0<s<t\). Equivalently,
		\[
		\Phi(x,\lambda t)
		\le
		C\lambda^q\Phi(x,t)
		\]
		for all \(\lambda\ge1\) and all \(t\ge0\).
	\end{definition}

	\begin{definition}[Standing structural assumption]
		Let \(\varphi\in\Phi_w(\Omega)\) and let \(\omega\) be a weight on \(\Omega\). We say that
		\((\varphi,\omega)\) satisfies \((\mathcal H_\Omega)\) if the following
		conditions hold.
		
		\smallskip
		
		\noindent\((H1)\) The pointwise lower index \(p_\varphi\) exists for a.e. \(x\in\Omega\),
		is measurable and satisfies
		\[
		1<p^-_\varphi\le p^+_\varphi<\infty,
		\qquad
		p_\varphi\in P^{\log}(\Omega),
		\]
		and we fix a whole-space log-H\"older extension
		\[
		\widetilde p_\varphi\in P^{\log}(\mathbb R^n),
		\qquad
		\widetilde p_\varphi=p_\varphi\quad\text{a.e. on }\Omega.
		\]
		
		\smallskip
		
		\noindent\((H2)\) The function \(\varphi\) satisfies
		\[
		(aInc)_{p_\varphi(\cdot)}.
		\]
		
		\smallskip
		
		\noindent\((H3)\) The function
		\[
		\psi_\varphi(x,t)=\varphi(x,t)^{1/p_\varphi(x)}
		\]
		belongs to \(\Phi_w(\Omega)\) and satisfies
		\[
		(A0)^\Omega,
		\qquad
		(A1)^\Omega_{\omega,p_\varphi},
		\qquad
		(A2)^\Omega_\omega .
		\]
	\end{definition}

	\begin{remark}
	For the normalized function
	\[
	\psi_\varphi(x,t)=\varphi(x,t)^{1/p_\varphi(x)},
	\]
	with \(0<p_\varphi(x)<\infty\), the normalization condition \((A0)^\Omega\)
	is equivalent for \(\psi_\varphi\) and \(\varphi\). Indeed, for each fixed
	\(t>0\),
	\[
	\psi_\varphi(x,t)\le1
	\quad\Longleftrightarrow\quad
	\varphi(x,t)\le1,
	\]
	and similarly with \(\ge1\).
	
	The same equivalence is not asserted for \((A1)\) and \((A2)\). These
	conditions compare generalized inverses at prescribed levels, while
	\[
	\psi_\varphi^{-1}(x,s)
	=
	\varphi^{-1}\left(x,s^{p_\varphi(x)}\right).
	\]
	Thus a fixed level \(s\) for \(\psi_\varphi\) corresponds to the
	\(x\)-dependent level \(s^{p_\varphi(x)}\) for \(\varphi\). For this
	reason the standing hypothesis imposes \((A1)\) and \((A2)\) directly on
	\(\psi_\varphi\), which is the function used in the lower-index reduction.
\end{remark}

\begin{remark}
The (aDec)-property is not a part of the standing hypotheses.
This property is used in the density theorem only to obtain the \(\Delta_2\)-type and
order-continuity properties of the weighted Musielak--Orlicz space.
Specifically, the condition \((aDec)_q\) is equivalent to
\[
\Phi(x,\lambda t)
\le
C\lambda^q\Phi(x,t),
\qquad \lambda\ge1.
\]
For comparison, we also note that the lower growth condition \((aInc)_{p_\varphi(\cdot)}\) is imposed on \(\varphi\) at the variable exponent
\(p_\varphi(\cdot)\), while the upper growth condition \((aDec)_{q_\psi}\) is imposed on the
normalization \(\psi_\varphi\) at a finite constant exponent \(q_\psi\).
\end{remark}

	\subsection{Whole-space extension hypothesis}

	\begin{definition}
		Assume that \((H1)\) holds, so that the lower-index exponent \(p_\varphi\) and
		its whole-space log-H\"older extension \(\widetilde p_\varphi\) have been fixed. We say that
		the weight \(\omega\) satisfies \((E_\Omega)_\omega\) if there exists a weight
		\[
		\widetilde\omega\in A_{\widetilde p_\varphi(\cdot)}(\mathbb R^n)
		\]
		such that \(\widetilde\omega=\omega\) a.e. on \(\Omega\).
	\end{definition}
	
	When \(\Omega=\mathbb R^n\), the condition \((E_\Omega)_\omega\)
	reduces to the usual whole-space condition
	\[
	\omega\in A_{p_\varphi(\cdot)}(\mathbb R^n).
	\]
	Observe that the condition \(\omega\in A_{p_\varphi(\cdot)}(\Omega)\) is induced from
	\((E_\Omega)_\omega\). The converse is not claimed. Thus \((E_\Omega)_\omega\) is a sufficient transfer hypothesis, not an intrinsic characterization of domain weights. It is immediate when \(\omega\) is the restriction of a known whole-space \(A_{\widetilde p_\varphi(\cdot)}\)-weight, including weights bounded above and below by positive constants. In the fixed-exponent case, further examples follow from classical weight-extension theorems under strengthened induced assumptions. For genuinely variable \(p_\varphi(\cdot)\) and arbitrary domains, this paper does not provide a general extension criterion; verifying \((E_\Omega)_\omega\) remains an explicit and potentially restrictive part of the domain hypotheses.

	\subsection{Main results}
	
	Let \(\Omega\subset\mathbb R^n\) be open and let \(\omega\) be a weight on \(\Omega\).
	Our main results center on the boundedness of the ambient-domain Hardy--Littlewood maximal
	function
	\[
	\mathcal M_\Omega:
	L^{\varphi(\cdot)}_\omega(\Omega)
	\longrightarrow
	L^{\varphi(\cdot)}_\omega(\Omega).
	\]
	
	The first result provides a lower-index sufficiency.
	
	\begin{theorem}\label{thm:sufficiency}
		Assume that \((\varphi,\omega)\) satisfies \((\mathcal H_\Omega)\) and that
		\(\omega\) satisfies \((E_\Omega)_\omega\). Then
		\(
		\mathcal M_\Omega
		\)
		is bounded.
	\end{theorem}

The second result gives necessity.
	
	\begin{theorem}\label{thm:associate-necessity}
			Let \(\omega\) be a weight and \(\varphi\in\Phi_w(\Omega)\).
			Assume that
			\(
			\chi_{B\cap\Omega}\in L^{\varphi(\cdot)}_\omega(\Omega)
			\)
			for all balls \(B\subset\mathbb R^n\) with
			\(B\cap\Omega\ne\emptyset\). If
			\(
			\mathcal M_\Omega
			\)
			is bounded, then
			\(
			\omega\in A_\varphi(\Omega).
			\)
	\end{theorem}

We note that the local characteristic-function assumption in Theorem \ref{thm:associate-necessity} follows from \((A0)^\Omega\) and the definition of a weight.

Combining Theorems \ref{thm:sufficiency} and \ref{thm:associate-necessity}, we obtain the following conditional whole-space characterization.
	
	\begin{corollary}[Conditional lower-index characterization]\label{cor:whole-space-sufficiency}
		Assume that \((\varphi,\omega)\) satisfies
		\((\mathcal H_{\mathbb R^n})\) and that
		\(L^{\varphi(\cdot)}_\omega(\mathbb R^n)\) satisfies the associate-to-lower-index product reduction.
		Then
		\[
		M:
		L^{\varphi(\cdot)}_\omega(\mathbb R^n)
		\longrightarrow
		L^{\varphi(\cdot)}_\omega(\mathbb R^n)
		\]
		is bounded if and only if \(\omega\in A_{p_\varphi(\cdot)}(\mathbb R^n)\).
	\end{corollary}

	\begin{remark}
		Corollary~\ref{cor:whole-space-sufficiency} is conditional on the
		associate-to-lower-index product-reduction hypothesis. In particular, it is not
		invoked for a general scalar Orlicz function unless that reduction has been
		established independently.
	\end{remark}

A summary is as follows.

	\begin{remark}
	The main results separate
	\begin{align*}
	(\mathcal H_\Omega)+(E_\Omega)_\omega
	&\quad\Longrightarrow\quad
	\mathcal M_\Omega\text{ bounded},
	\\
	\mathcal M_\Omega\text{ bounded}
	&\quad\Longrightarrow\quad
	\omega\in A_\varphi(\Omega).
	\end{align*}
	Under product reduction in Definition \ref{def:associate-lower-index-reduction},
	\[
	A_\varphi(\Omega)
	\quad\Longrightarrow\quad
	A_{p_\varphi(\cdot)}(\Omega).
	\]
	(See Proposition \ref{prop:associate-to-lower-index-ap}.)
	The second implication is understood under the local characteristic-function hypothesis of
	Theorem~\ref{thm:associate-necessity}. This leads to the characterization in
	Corollary~\ref{cor:whole-space-sufficiency}.
\end{remark}
	
	As an application, we derive the density of test functions in Musielak-Orlicz-Sobolev spaces.
	For \(k\in\mathbb N\), define
	\[
	W^{k,\varphi(\cdot)}_\omega(\mathbb R^n)
	:=
	\left\{
	u\in W^{k,1}_{\mathrm{loc}}(\mathbb R^n):
	D^\alpha u\in L^{\varphi(\cdot)}_\omega(\mathbb R^n)
	\text{ for all }|\alpha|\le k
	\right\},
	\]
	with the Luxemburg quasi-norm
	\[
	\|u\|_{W^{k,\varphi(\cdot)}_\omega(\mathbb R^n)}
	:=
	\sum_{|\alpha|\le k}
	\|D^\alpha u\|_{L^{\varphi(\cdot)}_\omega(\mathbb R^n)} .
	\]
	
	\begin{theorem}\label{thm:density-main}
		Assume that \((\varphi,\omega)\) satisfies \((\mathcal H_{\mathbb R^n})\), 
		\(
		\omega\in A_{p_\varphi(\cdot)}(\mathbb R^n)
		\)
		and \(\psi_\varphi\) satisfies \((aDec)_{q_\psi}\) for some \(q_\psi<\infty\). Then,
		for every \(k\in\mathbb N\),
		\(
		C_c^\infty(\mathbb R^n)
		\)
		is dense in
		\(
		W^{k,\varphi(\cdot)}_\omega(\mathbb R^n).
		\)
	\end{theorem}

	\section{Preliminaries}\label{sec:preliminaries}
	
	In this section we collect the notation and auxiliary facts used throughout the paper.
	The main structural assumptions were stated in Section~\ref{sec:framework-main} in the generalized Orlicz notation
	\((A0)\), \((A1)\), \((A2)\), \((aInc)\) and \((aDec)\). Here we recall the underlying
	variable exponent, generalized inverse, Musielak--Orlicz, complementary-function and
	K\"othe-associate facts needed in the proofs. The only change from the usual whole-space
	notation is that all ball tests are understood in the ambient-domain sense: a ball is a ball
	in \(\mathbb R^n\), while functions are integrated only on its intersection with \(\Omega\).
	
	\subsection{Basic notation}
	
	Throughout the paper \(\Omega\subset\mathbb R^n\) is an open set. A ball is always a
	Euclidean ball in \(\mathbb R^n\). For a ball \(B\), we write
	\[
	B_\Omega:=B\cap\Omega.
	\]
	The ambient ball basis is
	\[
	\mathcal B_\Omega
	:=
	\{B:\,B\text{ is a Euclidean ball in }\mathbb R^n,
	\ B_\Omega\ne\emptyset\}.
	\]
	For functions defined on \(\Omega\), integrals over \(B_\Omega\) are always taken with
	respect to Lebesgue measure on \(\Omega\). If \(\widetilde f=f\chi_\Omega\) is the zero
	extension of \(f\) to \(\mathbb R^n\), then
	\[
	\frac1{|B|}\int_{B_\Omega}|f(y)|\,dy
	=
	\fint_B |\widetilde f(y)|\,dy .
	\]
	The characteristic function of a measurable set \(E\) is denoted by \(\chi_E\).
	
	Weighted modulars are taken with respect to \(d\mu_\omega(x):=\omega(x)\,dx\), but
	K\"othe associate spaces are always taken with respect to Lebesgue measure. We write
	\(A\lesssim B\) if \(A\le CB\), and \(A\simeq B\) if \(A\lesssim B\) and \(B\lesssim A\).
	
	\subsection{Variable exponent spaces}
	
	Let \(p:\Omega\to[1,\infty]\) be measurable. 
	For each $E \subset \Omega$, we write
	\[
	p^-_E:=\operatorname*{ess\,inf}_{x\in E}p(x),
	\qquad
	p^+_E:=\operatorname*{ess\,sup}_{x\in E}p(x).
	\]
	When \(E=\Omega\), we write simply \(p^-\) and \(p^+\). The class \(P(\Omega)\) consists
	of all measurable exponents satisfying
	\[
	1<p^-\le p^+<\infty.
	\]
	The conjugate exponent is
	\[
	p'(x):=\frac{p(x)}{p(x)-1}.
	\]
	
	For a weight \(\omega\), the weighted variable exponent modular is
	\[
	\rho_{p(\cdot),\omega}(f)
	:=
	\int_\Omega |f(x)|^{p(x)}\omega(x)\,dx .
	\]
	The corresponding Luxemburg norm is
	\[
	\|f\|_{L^{p(\cdot)}_\omega(\Omega)}
	:=
	\inf\left\{
	\lambda>0:
	\rho_{p(\cdot),\omega}\!\left(\frac{f}{\lambda}\right)\le1
	\right\}.
	\]
	
	We record a weighted variable-exponent H\"older inequality.
	
	\begin{lemma}\label{lem:weighted-ve-holder}
		Let \(p\in P(\Omega)\).
		Let \(\omega\) be a weight. Then
		\[
		\int_\Omega |f(x)g(x)|\,dx
		\le
		C
		\|f\|_{L^{p(\cdot)}_\omega(\Omega)}
		\|\omega^{-1/p(\cdot)}g\|_{L^{p'(\cdot)}(\Omega)}
		\]
		for all measurable \(f,g\), where \(C\) depends only on \(p^-\) and \(p^+\).
	\end{lemma}
	
	\begin{proof}
		Write
		\[
		|f(x)g(x)|
		=
		|f(x)|\omega(x)^{1/p(x)}
		\cdot
		|g(x)|\omega(x)^{-1/p(x)}.
		\]
		The assertion follows from the usual variable-exponent H\"older inequality.
	\end{proof}
	
	We also recall the log-H\"older class. An exponent \(p\in P(\Omega)\) belongs to
	\(P^{\log}(\Omega)\) if there exists \(C>0\) such that
	\[
	|p(x)-p(y)|
	\le
	\frac{C}{\log(e+1/|x-y|)}
	\]
	for all \(x,y\in\Omega\) with \(x\ne y\), and when \(\Omega\) is unbounded, if there
	exists \(p_\infty\in\mathbb R\) such that
	\[
	|p(x)-p_\infty|
	\le
	\frac{C}{\log(e+|x|)}
	\]
	for all \(x\in\Omega\).

	\subsection{Consequences of the structural conditions}
	
	The structural assumptions \((A0)\), \((A1)\) and \((A2)\) were stated in Section~\ref{sec:framework-main} in
	inverse form. We record the direct forms used in the local maximal estimate.
	We start with some auxiliary generalized inverse estimates.

		\begin{lemma} \label{lem:generalized-inverse}
Let \(\varphi\in\Phi_w(\Omega)\). Then, for a.e. \(x\in\Omega\), the generalized
inverse \(\varphi^{-1}(x,\cdot)\) is increasing and left-continuous and
\[
\varphi^{-1}(x,\varphi(x,t))\le t
\]
for all \(t\ge0\). If \(0<\varphi(x,t)<\infty\), then
\[
\varphi^{-1}(x,\varphi(x,t))\simeq t,
\]
with constants depending only on the weak \(\Phi\)-data. Uniformly equivalent weak
\(\Phi\)-functions have uniformly comparable generalized inverses. Moreover, every
generalized weak \(\Phi\)-function has an equivalent generalized strong
\(\Phi\)-representative \(\widehat\varphi\), for which
\[
\widehat\varphi\bigl(x,\widehat\varphi^{-1}(x,\tau)\bigr)=\tau
\]
for a.e. \(x\in\Omega\) and all \(\tau\ge0\).
\end{lemma}

\begin{proof}
The first assertions follow pointwise from
\cite[Lemma 2.3.9(a), (c) and (d)]{HarjulehtoHasto2019}; the constants in the
comparison are uniform because the weak \(\Phi\)-data are uniform in \(x\).
Comparability of inverses under uniform equivalence follows pointwise from
\cite[Theorem 2.3.6]{HarjulehtoHasto2019}. Finally,
\cite[Theorem 2.5.10]{HarjulehtoHasto2019} gives an equivalent generalized strong
representative, and the displayed identity follows pointwise from
\cite[Lemma 2.3.3]{HarjulehtoHasto2019}.
\end{proof}

The next lemma presents a direct form of \((A1)^\Omega_{\omega,p}\).

	\begin{lemma}

		Assume that \(\psi_\varphi\) satisfies \((A1)^\Omega_{\omega,p}\). Then there exist
		\(\gamma_1\in(0,1]\) and \(C\ge1\) such that, for every ball \(B\in\mathcal B_\Omega\) with
		\[
		\|\chi_{B_\Omega}\|_{L^{p(\cdot)}_\omega(\Omega)}\le1,
		\]
		for a.e. \(x,y\in B_\Omega\) and for every \(t>0\) satisfying
		\[
		1\le \psi_\varphi(y,t)
		\le
		\|\chi_{B_\Omega}\|_{L^{p(\cdot)}_\omega(\Omega)}^{-1},
		\]
		one has
		\[
		\psi_\varphi(x,\gamma_1t)
		\le
		C\psi_\varphi(y,t).
		\]
		After decreasing \(\gamma_1\), the constant \(C\) can be absorbed.
	\end{lemma}
	
	\begin{proof}
Set \(s:=\psi_\varphi(y,t)\). By Lemma~\ref{lem:generalized-inverse}, there is a
structural constant \(c_0\in(0,1]\) such that
\[
c_0t\le \psi_\varphi^{-1}(y,s).
\]
The inverse form of \((A1)^\Omega_{\omega,p}\) therefore gives
\[
\beta_1c_0t
\le
\psi_\varphi^{-1}(x,s).
\]
With \(\gamma_1:=\beta_1c_0/2\), we have
\(\gamma_1t<\psi_\varphi^{-1}(x,s)\) whenever \(t>0\). By the definition of the
generalized inverse,
\[
\psi_\varphi(x,\gamma_1t)<s=\psi_\varphi(y,t).
\]
Thus the assertion holds, in fact, with \(C=1\); the final absorption statement is
therefore automatic.
\end{proof}

A direct form of \((A2)^\Omega_\omega\) is as follows.
	
	\begin{lemma}

		Assume that \(\psi_\varphi\) satisfies \((A2)^\Omega_\omega\). Then, for every \(s>0\),
		there exist \(\gamma_s\in(0,1]\), \(C_s\ge1\) and
		\[
		h_s\in L^1_\omega(\Omega)\cap L^\infty(\Omega),
		\qquad h_s\ge0,
		\]
		such that, for a.e. \(x,y\in\Omega\) and every \(t>0\) satisfying
		\[
		0\le \psi_\varphi(y,t)\le s,
		\]
		one has
		\[
		\psi_\varphi(x,\gamma_s t)
		\le
		C_s\{\psi_\varphi(y,t)+h_s(x)+h_s(y)\}.
		\]
	\end{lemma}
	
	\begin{proof}
Fix the level \(s+1\) in Definition~\ref{def:A2-weighted}.
Let
\(\beta_{s+1}\) and \(h_{s+1}\) be the corresponding data. Set
\(\tau:=\psi_\varphi(y,t)\). For every \(0<\varepsilon<1\),
\cite[Lemma 2.3.9(c)]{HarjulehtoHasto2019} gives
\[
t\le \psi_\varphi^{-1}(y,\tau+\varepsilon).
\]
Since \(0\le\tau+\varepsilon\le s+1\), the inverse form of
\((A2)^\Omega_\omega\) yields
\[
\beta_{s+1}t
\le
\psi_\varphi^{-1}\!\left(
x,\tau+\varepsilon+h_{s+1}(x)+h_{s+1}(y)
\right).
\]
Put \(\gamma_s:=\beta_{s+1}/2\). If \(t>0\), the left-hand side with
\(\gamma_s\) is strictly below the displayed inverse; hence, by the definition of the
generalized inverse,
\[
\psi_\varphi(x,\gamma_st)
<
\tau+\varepsilon+h_{s+1}(x)+h_{s+1}(y).
\]
The case \(t=0\) is immediate. Letting \(\varepsilon\downarrow0\) proves the claim
with \(C_s=1\) and \(h_s:=h_{s+1}\).
\end{proof}

For completeness, we note an instance verifying \((A1)\) condition.
	
	\begin{lemma}
Let \(p\in P(\Omega)\).
Let \(\omega\) be a weight and \(\psi\in\Phi_w(\Omega)\). 
Assume, in addition, that
\(t\mapsto\psi(x,t)\) is continuous for a.e. \(x\in\Omega\). Suppose that there
exist constants \(\gamma\in(0,1]\), \(C_0\ge1\) and \(C_1\ge1\) with the
following property: for every ball \(B\in\mathcal B_\Omega\) with
\[
\|\chi_{B_\Omega}\|_{L^{p(\cdot)}_\omega(\Omega)}\le1,
\]
for a.e. \(x,y\in B_\Omega\) and for every \(t>0\) such that
\[
1\le \psi(y,t)
\le
C_0\|\chi_{B_\Omega}\|_{L^{p(\cdot)}_\omega(\Omega)}^{-1},
\]
one has
\[
\psi(x,\gamma t)\le C_1\psi(y,t).
\]
Then \(\psi\) satisfies the inverse condition \((A1)^\Omega_{\omega,p}\), with
constants depending only on \(\gamma\), \(C_1\) and the weak \(\Phi\)-data.
\end{lemma}

\begin{proof}
Set
\[
N_B:=\|\chi_{B_\Omega}\|_{L^{p(\cdot)}_\omega(\Omega)}^{-1}.
\]
Fix \(B\in\mathcal B_\Omega\) with \(N_B\ge1\), fix a.e.
\(x,y\in B_\Omega\) and let \(1\le s\le N_B\). By continuity, monotonicity,
and the endpoint limits of a weak \(\Phi\)-function,
\[
t_s:=\psi^{-1}(y,s)>0,
\qquad
\psi(y,t_s)=s.
\]
Since
\(s\le N_B\le C_0N_B\), the assumed direct comparison applies to \(t_s\) and gives
\[
\psi(x,\gamma t_s)\le C_1s.
\]
Let \(L\ge1\) be the uniform constant in the weak \(\Phi\)-estimate
\(\psi(x,\theta r)\le L\theta\psi(x,r)\) for \(0<\theta\le1\) and set
\(\alpha:=(2LC_1)^{-1}\). Then
\[
\psi(x,\alpha\gamma t_s)\le \frac{s}{2}<s.
\]
By the definition of the generalized inverse,
\[
\alpha\gamma\psi^{-1}(y,s)
=
\alpha\gamma t_s
\le
\psi^{-1}(x,s).
\]
This is \((A1)^\Omega_{\omega,p}\), with
\(\beta_1=\alpha\gamma\).
\end{proof}
	
	\subsection{Weighted Musielak--Orlicz spaces}
	
	Let \(\varphi\in\Phi_w(\Omega)\) and let \(\omega\) be a weight. The weighted
	Musielak--Orlicz modular is
	\[
	\rho_{\varphi,\omega}(f)
	:=
	\int_\Omega \varphi(x,|f(x)|)\omega(x)\,dx .
	\]
	The corresponding Luxemburg functional (a quasi-norm in the weak setting) is
	\[
	\|f\|_{L^{\varphi(\cdot)}_\omega(\Omega)}
	:=
	\inf\left\{
	\lambda>0:
	\rho_{\varphi,\omega}\!\left(\frac{f}{\lambda}\right)\le1
	\right\}.
	\]
	
	\begin{lemma}\label{lem:lower-index-modular-identity}
		Assume that \(p_\varphi\) is defined and set
		\[
		G_f(x):=\psi_\varphi(x,|f(x)|).
		\]
		Then
		\[
		\rho_{\varphi,\omega}(f)
		=
		\rho_{p_\varphi(\cdot),\omega}(G_f).
		\]
	\end{lemma}
	
	\begin{proof}
		By the definition of \(\psi_\varphi\),
		\[
		\varphi(x,|f(x)|)
		=
		\psi_\varphi(x,|f(x)|)^{p_\varphi(x)}.
		\]
		Multiplying by \(\omega(x)\) and integrating gives the identity.
	\end{proof}

	\subsection{Complementary functions and K\"othe associates}
	
	When \(t\mapsto\varphi(x,t)\) is convex for a.e. \(x\), its pointwise complementary
	function is
	\[
	\varphi^*(x,s):=\sup_{t\ge0}\{st-\varphi(x,t)\}.
	\]
	If the original weak \(\Phi\)-function is not convex, we use an equivalent convex
	representative whenever the complementary function is invoked.
	
	Let \(X\) be a Banach function lattice over \(\Omega\) with respect to Lebesgue measure.
	Its K\"othe associate \(X'\) is defined by
	\[
	\|g\|_{X'}
	:=
	\sup_{\|f\|_X\le1}
	\int_\Omega |f(x)g(x)|\,dx .
	\]
	In this paper the associate is always taken with respect to Lebesgue measure. This is
	essential because the maximal operator averages with respect to Lebesgue measure.
	
	The next lemma provides a localization of the associate norm.
	
	\begin{lemma}
		Let \(X\) be a Banach function lattice over \(\Omega\). For every measurable set
		\(E\subset\Omega\),
		\[
		\|\chi_E\|_{X'}
		=
		\sup_{\substack{\|f\|_X\le1\\ f\ge0,\ \operatorname{supp}f\subset E}}
		\int_E f(x)\,dx .
		\]
	\end{lemma}
	
	\begin{proof}
		By definition,
		\[
		\|\chi_E\|_{X'}
		=
		\sup_{\|f\|_X\le1}
		\int_E |f(x)|\,dx .
		\]
		For any admissible \(f\), the function \(|f|\chi_E\) is nonnegative, supported in \(E\),
		and satisfies \(\||f|\chi_E\|_X\le\|f\|_X\) by the lattice property. Hence the same
		supremum is obtained by restricting to nonnegative functions supported in \(E\).
	\end{proof}

Under suitable conditions, one also has the following associate representation.
	
	\begin{proposition}\label{prop:associate-representation}
		Assume that \(\varphi\) has an equivalent convex representative and that the standard
		Musielak--Orlicz duality theorem applies. Let
		\[
		X:=L^{\varphi(\cdot)}_\omega(\Omega).
		\]
		Then
		\[
		X'
		\simeq
		L^{\Phi_\omega^*(\cdot)}(\Omega),
		\qquad
		\Phi_\omega^*(x,s)
		=
		\omega(x)\varphi^*\!\left(x,\frac{s}{\omega(x)}\right).
		\]
		Equivalently,
		\[
		\|g\|_{X'}
		\simeq
		\|\omega^{-1}g\|_{L^{\varphi^*(\cdot)}_\omega(\Omega)} .
		\]
	\end{proposition}
	
	\begin{proof}
		The weighted modular is the unweighted modular generated by
		\(\Phi_\omega(x,t):=\omega(x)\varphi(x,t)\). Its complementary function is
		\[
		\Phi_\omega^*(x,s)
		=
		\omega(x)\varphi^*\!\left(x,\frac{s}{\omega(x)}\right).
		\]
		The standard Musielak--Orlicz duality theorem gives the first assertion and the displayed
		identity gives the equivalent weighted form.
	\end{proof}

In view of Proposition~\ref{prop:associate-representation}, \(\omega\in A_\varphi(\Omega)\) is equivalent to
\[
\sup_{B\subset\mathbb R^n}
\frac{
	\|\chi_{B\cap\Omega}\|_{L^{\varphi(\cdot)}_\omega(\Omega)}
	\|\omega^{-1}\chi_{B\cap\Omega}\|_{L^{\varphi^*(\cdot)}_\omega(\Omega)}
}{|B|}<\infty.
\]
This is the exact ambient-domain Musielak--Orlicz analogue of the
complementary-function condition in scalar Orlicz theory.
	
	\subsection{Banach-function and order-continuity properties}
	
\begin{proposition}
	\label{prop:ambient-bfs}
	Let \(\varphi\in\Phi_w(\Omega)\).
	Let \(\omega\) be a weight.
	Assume that \(\varphi\) has an equivalent convex representative. Assume that \(p_\varphi^+<\infty\) and that
	\(\psi_\varphi\) satisfies \((aDec)_{q_\psi}\) for some finite
	\(q_\psi\). Then \(L^{\varphi(\cdot)}_\omega(\Omega)\), equipped with an
	equivalent Luxemburg norm associated with a convex representative, is a Banach function
	lattice. Moreover, if \(0\le f_j\le F\),
	\(F\in L^{\varphi(\cdot)}_\omega(\Omega)\) and \(f_j\to0\) a.e.,
	then
	\[
	\|f_j\|_{L^{\varphi(\cdot)}_\omega(\Omega)}\to0.
	\]
\end{proposition}

\begin{proof}
	Passing to an equivalent convex representative gives a genuine
	Musielak--Orlicz modular with an equivalent Luxemburg norm. Since
	\[
	\varphi(x,t)=\psi_\varphi(x,t)^{p_\varphi(x)}
	\]
	and \(p_\varphi^+<\infty\), the assumption \((aDec)_{q_\psi}\)
	implies a finite \(\Delta_2\)-type estimate for \(\varphi\). Indeed,
	for \(\lambda\ge1\),
	\[
	\psi_\varphi(x,\lambda t)
	\le
	C\lambda^{q_\psi}\psi_\varphi(x,t),
	\]
	and therefore
	\[
	\varphi(x,\lambda t)
	\le
	C^{p_\varphi(x)}
	\lambda^{q_\psi p_\varphi(x)}
	\varphi(x,t)
	\le
	C^{p_\varphi^+}
	\lambda^{q_\psi p_\varphi^+}
	\varphi(x,t).
	\]
	This gives the Banach-lattice and modular--norm properties.
	
	If \(0\le f_j\le F\) and \(f_j\to0\) a.e., then, for every fixed
	\(\lambda>0\),
	\[
	0\le
	\varphi\left(x,\frac{f_j(x)}{\lambda}\right)
	\le
	\varphi\left(x,\frac{F(x)}{\lambda}\right).
	\]
	By the preceding \(\Delta_2\)-type estimate, the right-hand side is
	integrable with respect to \(\omega(x)\,dx\). Dominated convergence
	gives
	\[
	\rho_{\varphi,\omega}(f_j/\lambda)\to0
	\]
	for every \(\lambda>0\). Hence
	\[
	\|f_j\|_{L^{\varphi(\cdot)}_\omega(\Omega)}\to0
	\]
	as claimed.
\end{proof}

\begin{remark}
	Under the standing hypothesis \((\mathcal H_\Omega)\), the function
	\(\varphi\) has an equivalent convex representative and
	\(p_\varphi^+<\infty\). Indeed, the convexification follows from the
	almost-increasing lower growth of \(\varphi\); see
	\cite[Lemma 2.5.9]{HarjulehtoHasto2019}. If, in addition,
	\(\psi_\varphi\) satisfies \((aDec)_{q_\psi}\) for some finite
	\(q_\psi\), then all the assumptions of
	Proposition~\ref{prop:ambient-bfs} are satisfied.
\end{remark}

	\section{Weighted variable exponent estimates for the ambient operator}
	\label{sec:domain-ve-estimates}
	
	In this section we collect the weighted variable exponent estimates used in the proof of
	the Musielak--Orlicz maximal theorem. The key point is that the lower-index reduction
	passes through the weighted space
	\[
	L^{p(\cdot)}_\omega(\Omega),
	\qquad p=p_\varphi.
	\]
	Throughout this section \(p\in P^{\log}(\Omega)\) is a fixed exponent, supplied in
	applications by the lower-index exponent \(p_\varphi\) in \((H1)\) and \(\omega\) is a
	weight on \(\Omega\). When \((E_\Omega)_\omega\) is used in this section, it is understood
	relative to this fixed exponent and its chosen whole-space log-H\"older extension. We write
	\[
	B_\Omega:=B\cap\Omega
	\]
	for a Euclidean ball \(B\subset\mathbb R^n\).
	
	\subsection{The ambient weighted variable exponent maximal principle}
	
	The following theorem is the lower-index maximal input used in the sequel. It is the
	domain form of the weighted variable exponent maximal theorem applied to the zero extension
	of a function outside \(\Omega\). Equivalently, it follows from an admissible whole-space
	extension of \(p\) and \(\omega\).
	
	\begin{theorem}
		\label{thm:ambient-ve-maximal}
		Assume that
		\[
		p\in P^{\log}(\Omega),
		\qquad
		(E_\Omega)_\omega\text{ holds}
		\]
		relative to the chosen whole-space extension of \(p\). Then
		\[
		\mathcal M_\Omega:
		L^{p(\cdot)}_\omega(\Omega)
		\longrightarrow
		L^{p(\cdot)}_\omega(\Omega)
		\]
		is bounded. 
	\end{theorem}
	
	\begin{proof}
		Choose the fixed whole-space extension \(\widetilde p\in P^{\log}(\mathbb R^n)\) of \(p\) and the weight extension \(\widetilde\omega\in A_{\widetilde p(\cdot)}(\mathbb R^n)\) supplied by \((E_\Omega)_\omega\). If \(g\) is extended by zero outside
		\(\Omega\), then
		\[
		\mathcal M_\Omega g=M(g\chi_\Omega)|_\Omega .
		\]
		In the multiplier-weight notation of
		\cite{CruzUribeDieningHasto2011}, set
		\(v:=\widetilde\omega^{1/\widetilde p(\cdot)}\). Then
		\[
		\|F\|_{L^{\widetilde p(\cdot)}_{\widetilde\omega}(\mathbb R^n)}
		=
		\|Fv\|_{L^{\widetilde p(\cdot)}(\mathbb R^n)},
		\]
		and the measure-weight condition
		\(\widetilde\omega\in A_{\widetilde p(\cdot)}(\mathbb R^n)\) is exactly the
		corresponding multiplier-weight condition for \(v\). The ball and cube versions
		are equivalent in \(\mathbb R^n\). Hence the weighted variable exponent maximal
		theorem of Cruz-Uribe, Diening and H\"ast\"o in \cite{CruzUribeDieningHasto2011} gives
		\[
		\|M(g\chi_\Omega)\|_{L^{\widetilde p(\cdot)}_{\widetilde\omega}(\mathbb R^n)}
		\le
		C\|g\chi_\Omega\|_{L^{\widetilde p(\cdot)}_{\widetilde\omega}(\mathbb R^n)}.
		\]
		Restricting to \(\Omega\) gives the stated boundedness. 
	\end{proof}

	\begin{lemma}
		\label{lem:maximal-modular-consequence}
		Assume that \(p\in P^{\log}(\Omega)\) and that \((E_\Omega)_\omega\) holds. Let
		\[
		C_M:=
		\|\mathcal M_\Omega\|_{L^{p(\cdot)}_\omega(\Omega)\to L^{p(\cdot)}_\omega(\Omega)}.
		\]
		If \(\rho_{p(\cdot),\omega}(g)\le1\), then
		\[
		\rho_{p(\cdot),\omega}\left(\frac{\mathcal M_\Omega g}{C_M}\right)\le1.
		\]
		Consequently,
		\[
		\rho_{p(\cdot),\omega}(\mathcal M_\Omega g)\le C,
		\]
		where \(C\) depends only on \(p^+\) and \(C_M\).
	\end{lemma}
	
	\begin{proof}
		The assumption \(\rho_{p(\cdot),\omega}(g)\le1\) implies
		\(\|g\|_{L^{p(\cdot)}_\omega(\Omega)}\le1\). By
		Theorem~\ref{thm:ambient-ve-maximal},
		\[
		\|\mathcal M_\Omega g\|_{L^{p(\cdot)}_\omega(\Omega)}\le C_M .
		\]
		The definition of the Luxemburg norm gives
		\[
		\rho_{p(\cdot),\omega}\left(\frac{\mathcal M_\Omega g}{C_M}\right)\le1.
		\]
		Multiplying back yields
		\[
		\rho_{p(\cdot),\omega}(\mathcal M_\Omega g)
		\le
		\max\{C_M^{p^-},C_M^{p^+}\}
		\]
		as required.
	\end{proof}
	
	\subsection{The ambient \texorpdfstring{\(A_{p(\cdot)}\)}{A-p(.)}-condition and ball estimates}
	
	Recall that the ambient weighted Muckenhoupt condition is
	\[
	[\omega]_{A_{p(\cdot)}(\Omega)}
	=
	\sup_{B\subset\mathbb R^n}
	\frac{
		\|\chi_{B_\Omega}\|_{L^{p(\cdot)}_\omega(\Omega)}
		\|\omega^{-1/p(\cdot)}\chi_{B_\Omega}\|_{L^{p'(\cdot)}(\Omega)}
	}{|B|}
	<\infty .
	\]
	
	We record several useful estimates for $A_{p(\cdot)}$-weights.
	
	\begin{lemma}
		\label{lem:dual-indicator-estimate}
		Assume that \(\omega\in A_{p(\cdot)}(\Omega)\). Then, for every ball
		\(B\subset\mathbb R^n\) with \(B_\Omega\ne\emptyset\),
		\[
		\|\omega^{-1/p(\cdot)}\chi_{B_\Omega}\|_{L^{p'(\cdot)}(\Omega)}
		\le
		[\omega]_{A_{p(\cdot)}(\Omega)}
		\frac{|B|}{\|\chi_{B_\Omega}\|_{L^{p(\cdot)}_\omega(\Omega)}} .
		\]
	\end{lemma}
	
	\begin{proof}
		This is Definition~\ref{def:ambient-ap} rewritten in ambient notation.
	\end{proof}
	
	\begin{lemma}
		\label{lem:weighted-average-estimate}
		Assume that \(\omega\in A_{p(\cdot)}(\Omega)\). Then, for every ball
		\(B\subset\mathbb R^n\) with \(B_\Omega\ne\emptyset\) and every
		\(g\in L^{p(\cdot)}_\omega(\Omega)\),
		\[
		\frac1{|B|}\int_{B_\Omega}|g(y)|\,dy
		\le
		C
		\|g\|_{L^{p(\cdot)}_\omega(\Omega)}
		\|\chi_{B_\Omega}\|_{L^{p(\cdot)}_\omega(\Omega)}^{-1},
		\]
		where \(C\) depends only on \(p^-\), \(p^+\) and
		\([\omega]_{A_{p(\cdot)}(\Omega)}\).
	\end{lemma}
	
	\begin{proof}
		By Lemma~\ref{lem:weighted-ve-holder},
		\[
		\int_{B_\Omega}|g(y)|\,dy
		\le
		C
		\|g\|_{L^{p(\cdot)}_\omega(\Omega)}
		\|\omega^{-1/p(\cdot)}\chi_{B_\Omega}\|_{L^{p'(\cdot)}(\Omega)} .
		\]
		Using Lemma~\ref{lem:dual-indicator-estimate} and dividing by \(|B|\), we obtain the
		claim.
	\end{proof}
	
	\begin{corollary}
		\label{cor:normalized-lower-index-average}
		Assume that \(\omega\in A_{p(\cdot)}(\Omega)\). Let
		\(g\in L^{p(\cdot)}_\omega(\Omega)\) satisfy
		\(\rho_{p(\cdot),\omega}(g)\le1\). Then, for every ball
		\(B\subset\mathbb R^n\) with \(B_\Omega\ne\emptyset\),
		\[
		\frac1{|B|}\int_{B_\Omega}|g(y)|\,dy
		\le
		C
		\|\chi_{B_\Omega}\|_{L^{p(\cdot)}_\omega(\Omega)}^{-1}.
		\]
	\end{corollary}
	
	\begin{proof}
		The modular assumption gives \(\|g\|_{L^{p(\cdot)}_\omega(\Omega)}\le1\). The claim
		follows from Lemma~\ref{lem:weighted-average-estimate}.
	\end{proof}
	
	\subsection{Restriction from the whole space}

The next whole-space weighted variable exponent result is known.
See \cite{CruzUribeDieningHasto2011}.

	\begin{proposition}
		\label{cor:whole-space-ve-maximal}
		Let \(\Omega=\mathbb R^n\). Assume that \(p\in P^{\log}(\mathbb R^n)\) and
		\(\omega\in A_{p(\cdot)}(\mathbb R^n)\). Then
		\[
		M:
		L^{p(\cdot)}_\omega(\mathbb R^n)
		\longrightarrow
		L^{p(\cdot)}_\omega(\mathbb R^n)
		\]
		is bounded.
	\end{proposition}

The next elementary observation details the zero-extension mechanism behind the ambient
operator. It is the mechanism by which the maximal estimate is reduced to the whole-space theorem.
	
	\begin{corollary}

		Assume that \(p\in P(\Omega)\) and \(\omega\) admit extensions
		\[
		\widetilde p\in P^{\log}(\mathbb R^n),
		\qquad
		\widetilde\omega\in A_{\widetilde p(\cdot)}(\mathbb R^n),
		\]
		with \(\widetilde p|_\Omega=p\) and \(\widetilde\omega|_\Omega=\omega\). Then
		\(\mathcal M_\Omega\) is bounded on \(L^{p(\cdot)}_\omega(\Omega)\).
	\end{corollary}
	
	\begin{proof}
		
		Let \(f\in L^{p(\cdot)}_\omega(\Omega)\) and \(\widetilde f=f\chi_\Omega\). Then
		\[
		\|\widetilde f\|_{L^{\widetilde p(\cdot)}_{\widetilde\omega}(\mathbb R^n)}
		=
		\|f\|_{L^{p(\cdot)}_\omega(\Omega)} .
		\]
		For every \(x\in\Omega\),
		\[
		\mathcal M_\Omega f(x)=M\widetilde f(x).
		\]
		The claim now follows from Proposition~\ref{cor:whole-space-ve-maximal}.
	\end{proof}
	
	\subsection{Application to the lower-index normalization}
	
	Let
	\[
	p=p_\varphi,
	\qquad
	\psi=\psi_\varphi,
	\qquad
	G_f(x)=\psi(x,|f(x)|).
	\]
	By Lemma~\ref{lem:lower-index-modular-identity},
	\[
	\rho_{\varphi,\omega}(f)
	=
	\rho_{p(\cdot),\omega}(G_f).
	\]

	\begin{lemma}
		\label{lem:lower-index-normalized-average}
		Assume that \(p=p_\varphi\in P(\Omega)\) and
		\(\omega\in A_{p(\cdot)}(\Omega)\). If
		\(\rho_{\varphi,\omega}(f)\le1\), then, for every ball
		\(B\subset\mathbb R^n\) with \(B_\Omega\ne\emptyset\),
		\[
		\frac1{|B|}\int_{B_\Omega}\psi(y,|f(y)|)\,dy
		\le
		C
		\|\chi_{B_\Omega}\|_{L^{p(\cdot)}_\omega(\Omega)}^{-1}.
		\]
	\end{lemma}
	
	\begin{proof}
		Set \(G_f(y):=\psi(y,|f(y)|)\). The modular identity gives
		\(\rho_{p(\cdot),\omega}(G_f)\le1\). The result follows from
		Corollary~\ref{cor:normalized-lower-index-average}.
	\end{proof}
	
	\begin{lemma}
		\label{lem:lower-index-maximal-modular}
		Assume that \(p=p_\varphi\in P^{\log}(\Omega)\), that \((E_\Omega)_\omega\) holds and that
		\(\rho_{\varphi,\omega}(f)\le1\). Then
		\[
		\rho_{p(\cdot),\omega}\bigl(\mathcal M_\Omega G_f\bigr)\le C,
		\qquad
		G_f(x)=\psi_\varphi(x,|f(x)|).
		\]
	\end{lemma}
	
	\begin{proof}
		The lower-index modular identity gives \(\rho_{p(\cdot),\omega}(G_f)\le1\). The conclusion follows from Lemma~\ref{lem:maximal-modular-consequence}.
	\end{proof}
	

	\section{Local normalized Musielak--Orlicz estimates}
	\label{sec:local-mo-estimates}
	
	In this section we prove the local estimate which reduces the weighted
	Musielak--Orlicz maximal problem to an ambient weighted variable exponent maximal estimate at the
	lower-index scale. Throughout the section we write
	\[
	p:=p_\varphi,
	\qquad
	\psi:=\psi_\varphi .
	\]
	Thus
	\[
	\varphi(x,t)=\psi(x,t)^{p(x)} .
	\]
	The assumptions used in this section are the sufficiency assumptions
	\[
	(\varphi,\omega)\text{ satisfies }(\mathcal H_\Omega),
	\qquad
	\omega\in A_{p(\cdot)}(\Omega).
	\]
	The proof follows the standard strategy: normalize by the lower index, estimate the normalized
	function through the ambient weighted variable exponent maximal operator and transfer the estimate
	back to \(\varphi\) using the ambient \((A1)\) and \((A2)\) assumptions. We work with generalized inverses.
	
	\subsection{The normalized input}
	
	For a measurable function \(f\), define
	\[
	G_f(x):=\psi(x,|f(x)|).
	\]
	Then
	\[
	\rho_{\varphi,\omega}(f)=\rho_{p(\cdot),\omega}(G_f).
	\]
	
	\begin{lemma}
		\label{lem:normalized-input}
		If \(\rho_{\varphi,\omega}(f)\le1\), then
		\[
		\rho_{p(\cdot),\omega}(G_f)\le1,
		\qquad
		\|G_f\|_{L^{p(\cdot)}_\omega(\Omega)}\le1.
		\]
	\end{lemma}
	
	\begin{proof}
		The modular identity is Lemma~\ref{lem:lower-index-modular-identity}. The norm estimate
		follows from the Luxemburg definition.
	\end{proof}
	
	\begin{lemma}
		\label{lem:normalized-average-bound}
		Assume that \(\omega\in A_{p(\cdot)}(\Omega)\). If
		\(\rho_{\varphi,\omega}(f)\le1\), then, for every ball \(B\subset\mathbb R^n\) with \(B_\Omega\ne\emptyset\),
		\[
		\frac1{|B|}\int_{B_\Omega} G_f(y)\,dy
		\le
		C
		\|\chi_{B_\Omega}\|_{L^{p(\cdot)}_\omega(\Omega)}^{-1}.
		\]
	\end{lemma}
	
	\begin{proof}
		This is Lemma~\ref{lem:lower-index-normalized-average}. Equivalently, it follows from
		Lemma~\ref{lem:weighted-ve-holder}, Lemma~\ref{lem:normalized-input} and the definition
		of \(A_{p(\cdot)}(\Omega)\).
	\end{proof}
	
	\subsection{Local inverse comparison}
	
	For a ball \(B\subset\mathbb R^n\) with \(B_\Omega\ne\emptyset\), set
	\[
	N_B:=\|\chi_{B_\Omega}\|_{L^{p(\cdot)}_\omega(\Omega)}^{-1},
	\qquad
	L_B:=\max\{1,N_B\}.
	\]
	The number \(N_B\) is the natural normalized size of the ambient ball in the weighted variable
	exponent scale. The condition \((A1)^\Omega_{\omega,p}\) is imposed precisely on the range
	\[
	1\le \psi(y,t)\le N_B
	\]
	when \(N_B\ge1\), while \((A2)^\Omega_\omega\) handles bounded levels.
	
	The following elementary inverse calculus is useful later.
	
\begin{lemma}
	\label{lem:inverse-calculus-section4}
	Let \(\psi\in\Phi_w(\Omega)\). Then there exists \(C\ge1\), depending
	only on the weak \(\Phi\)-data, such that, for a.e. \(x\in\Omega\),
	for all \(a,b\ge0\), \(\theta\in[0,1]\) and \(\lambda\ge1\),
	\[
	\psi^{-1}(x,a)+\psi^{-1}(x,b)
	\le
	C\psi^{-1}(x,a+b),
	\]
	\[
	\theta\psi^{-1}(x,a)
	\le
	C\psi^{-1}(x,\theta a),
	\]
	and
	\[
	\psi^{-1}(x,\lambda a)
	\le
	C\lambda\psi^{-1}(x,a).
	\]
\end{lemma}

\begin{proof}
For a.e. fixed \(x\), \cite[Proposition~2.3.7(a)]{HarjulehtoHasto2019} shows that
\(g_x:=\psi^{-1}(x,\cdot)\) satisfies \((aDec)_1\), while
\cite[Lemma 2.3.9(a)]{HarjulehtoHasto2019} shows that \(g_x\) is increasing.
Hence
\[
g_x(a)+g_x(b)\le 2g_x(a+b).
\]
Moreover, the \((aDec)_1\)-property gives, uniformly in \(x\),
\[
\theta g_x(a)\le Cg_x(\theta a),
\qquad
g_x(\lambda a)\le C\lambda g_x(a),
\]
for \(0<\theta\le1\) and \(\lambda\ge1\). The endpoint cases are immediate.
\end{proof}

	The next lemma is the technical bridge from the structural assumptions to the local
averaging estimate. It uses the inverse forms of \((A1)\) and \((A2)\) directly,
together with the inverse growth estimates in Lemma~\ref{lem:inverse-calculus-section4}.

\begin{lemma}
		\label{lem:local-inverse-comparison}
		Assume that \((\varphi,\omega)\) satisfies \((\mathcal H_\Omega)\). For every \(K\ge1\)
		there exist constants \(\theta_K\in(0,1]\), \(C_K\ge1\) and a function
		\[
		h_K\in L^1_\omega(\Omega)\cap L^\infty(\Omega),
		\qquad h_K\ge0,
		\]
		such that, for every ball \(B\subset\mathbb R^n\) with \(B_\Omega\ne\emptyset\), for a.e. \(x,y\in B_\Omega\) and for every
		\(0\le s\le K L_B\),
		\[
		\theta_K\psi^{-1}(y,s)
		\le
		\psi^{-1}\!\left(
		x,
		C_K\{s+h_K(x)+h_K(y)\}
		\right).
		\]
		Moreover, on the normalized non-bounded range \(K<s\le K N_B\), when \(N_B>1\), the same
		estimate holds with \(h_K\equiv0\), with a possibly larger \(C_K\).
	\end{lemma}
	
	\begin{proof}
Fix \(K\ge1\). For the bounded range \(0\le s\le K\), apply
\((A2)^\Omega_\omega\) at level \(K\). Thus there exist
\(\beta_K^{(2)}\in(0,1]\) and
\(h_K\in L^1_\omega(\Omega)\cap L^\infty(\Omega)\), \(h_K\ge0\), such that
\[
\beta_K^{(2)}\psi^{-1}(y,s)
\le
\psi^{-1}\!\left(x,s+h_K(x)+h_K(y)\right).
\]
This gives the required estimate on the bounded range.

Now suppose that \(K<s\le KN_B\). Then \(N_B>1\) and
\(r:=s/K\) satisfies \(1<r\le N_B\). By
Lemma~\ref{lem:inverse-calculus-section4},
\[
\psi^{-1}(y,s)
=
\psi^{-1}(y,Kr)
\le
C K\psi^{-1}(y,r).
\]
The inverse condition \((A1)^\Omega_{\omega,p}\) gives
\[
\beta_1\psi^{-1}(y,r)
\le
\psi^{-1}(x,r)
\le
\psi^{-1}(x,s).
\]
Consequently,
\[
\frac{\beta_1}{CK}\psi^{-1}(y,s)
\le
\psi^{-1}(x,s),
\]
which is the asserted no-\(h_K\) estimate. Taking
\(\theta_K\) to be the smaller of \(\beta_K^{(2)}\) and
\(\beta_1/(CK)\), and enlarging \(C_K\) if necessary, proves both statements.
\end{proof}
	
	\subsection{A local inverse-Jensen estimate}
	
	\begin{lemma}
		\label{lem:local-inverse-jensen}
		Assume that \((\varphi,\omega)\) satisfies \((\mathcal H_\Omega)\). Let \(K\ge1\). Then
		there exist \(\beta_K\in(0,1]\), \(C_K\ge1\) and
		\[
		h_K\in L^1_\omega(\Omega)\cap L^\infty(\Omega),
		\qquad h_K\ge0,
		\]
		such that the following holds. Let \(B\subset\mathbb R^n\) satisfy \(B_\Omega\ne\emptyset\), let \(u\ge0\) be measurable on
		\(B_\Omega\) and assume
		\[
		A_B:=\frac1{|B|}\int_{B_\Omega} \psi(y,u(y))\,dy\le K L_B .
		\]
		Then, for a.e. \(x\in B_\Omega\),
		\[
		\psi\!\left(x,\beta_K\frac1{|B|}\int_{B_\Omega} u(y)\,dy\right)
		\le
		C_K A_B+C_Kh_K(x)+C_K\frac1{|B|}\int_{B_\Omega} h_K(y)\,dy .
		\]
	\end{lemma}
	
	\begin{proof}
Fix \(x\in B_\Omega\) outside the exceptional null set in
Lemma~\ref{lem:local-inverse-comparison}. Set
\(\tau(y):=\psi(y,u(y))\). Since \(A_B<\infty\), we may disregard a null set on
which \(\tau=\infty\). Split
\[
\begin{aligned}
Z&:=\{y\in B_\Omega:\tau(y)=0\},\\
E_0&:=\{y\in B_\Omega:0<\tau(y)\le K\},\\
E_1&:=\{y\in B_\Omega:K<\tau(y)\le K L_B\},\\
E_2&:=\{y\in B_\Omega:\tau(y)>K L_B\}.
\end{aligned}
\]
We shall repeatedly use the following quasi-Jensen consequence of
Lemma~\ref{lem:inverse-calculus-section4}: if \(g=\psi^{-1}(x,\cdot)\),
\(E\subset B_\Omega\) is measurable and \(R:E\to[0,\infty)\) is measurable, then
\[
\frac1{|B|}\int_E g(R(y))\,dy
\le
C g\!\left(\frac1{|B|}\int_E R(y)\,dy\right).
\]
Indeed, with \(\overline R=|B|^{-1}\int_E R\), split \(E\) into
\(\{R\le\overline R\}\) and \(\{R>\overline R\}\), then use monotonicity on the
first set and \((aDec)_1\) on the second. The case \(\overline R=0\) is immediate.

Let \(0<\varepsilon\le K\). On \(Z\),
\cite[Lemma 2.3.9(c)]{HarjulehtoHasto2019} gives
\[
u(y)\le\psi^{-1}(y,\varepsilon).
\]
The bounded-level part of Lemma~\ref{lem:local-inverse-comparison} and
the quasi-Jensen estimate therefore yield
\[
\theta_K\frac1{|B|}\int_Zu(y)\,dy
\le
C\psi^{-1}\!\left(
x,C_K\left\{\varepsilon+h_K(x)+\frac1{|B|}\int_{B_\Omega}h_K(y)\,dy\right\}
\right).
\]

On \(E_0\), Lemma~\ref{lem:generalized-inverse} gives
\(u(y)\le C\psi^{-1}(y,\tau(y))\). Hence the bounded-level part of
Lemma~\ref{lem:local-inverse-comparison}, followed by
the quasi-Jensen estimate, gives
\[
\theta_K\frac1{|B|}\int_{E_0}u(y)\,dy
\le
C\psi^{-1}\!\left(
x,C_K\left\{A_B+h_K(x)+\frac1{|B|}\int_{B_\Omega}h_K(y)\,dy\right\}
\right).
\]

On \(E_1\), the same inverse-composition estimate and the no-\(h_K\) part of
Lemma~\ref{lem:local-inverse-comparison} give, using
the quasi-Jensen estimate,
\[
\theta_K\frac1{|B|}\int_{E_1}u(y)\,dy
\le
C\psi^{-1}(x,C_KA_B).
\]

For \(E_2\), Lemma~\ref{lem:generalized-inverse} and
Lemma~\ref{lem:inverse-calculus-section4} give
\[
u(y)
\le
C\psi^{-1}(y,\tau(y))
\le
C\frac{\tau(y)}{K L_B}\psi^{-1}(y,K L_B).
\]
Applying Lemma~\ref{lem:local-inverse-comparison} at level \(K L_B\), using the
no-\(h_K\) part when \(L_B>1\), gives
\[
\theta_Ku(y)
\le
C\frac{\tau(y)}{K L_B}\psi^{-1}(x,C_KK L_B).
\]
When \(L_B=1\), the bounded-level part gives the same estimate after changing
\(C_K\), because \(K\ge1\) and \(h_K\in L^\infty(\Omega)\). Therefore
\[
\theta_K\frac1{|B|}\int_{E_2}u(y)\,dy
\le
C\frac{A_B}{K L_B}\psi^{-1}(x,C_KK L_B)
\le
C\psi^{-1}(x,C_KA_B),
\]
where the last step follows from Lemma~\ref{lem:inverse-calculus-section4} and
\(A_B\le K L_B\).

Combining the four estimates and using the first inequality in
Lemma~\ref{lem:inverse-calculus-section4}, we obtain a constant
\(\beta_K^{(0)}\in(0,1]\) such that
\[
\beta_K^{(0)}\frac1{|B|}\int_{B_\Omega}u(y)\,dy
\le
\psi^{-1}\!\left(
x,C_K\left\{A_B+h_K(x)+\frac1{|B|}\int_{B_\Omega}h_K(y)\,dy+\varepsilon\right\}
\right).
\]
Set \(\beta_K:=\beta_K^{(0)}/2\). If the inverse on the right is positive, the
left-hand side with \(\beta_K\) is strictly smaller than that inverse; if it is zero,
the left-hand side is zero. The definition of the generalized inverse therefore gives
\[
\psi\!\left(x,\beta_K\frac1{|B|}\int_{B_\Omega}u(y)\,dy\right)
\le
C_K\left\{A_B+h_K(x)+\frac1{|B|}\int_{B_\Omega}h_K(y)\,dy+\varepsilon\right\}.
\]
Letting \(\varepsilon\downarrow0\) proves the lemma.
\end{proof}
	
	\subsection{The local normalized estimate}
	
	\begin{proposition}
		\label{prop:local-normalized-estimate}
		Assume that \((\varphi,\omega)\) satisfies \((\mathcal H_\Omega)\) and that
		\(\omega\in A_{p(\cdot)}(\Omega)\). Then there exist \(\beta\in(0,1]\), \(C\ge1\) and
		\[
		h\in L^1_\omega(\Omega)\cap L^\infty(\Omega),
		\qquad h\ge0,
		\]
		such that, if \(\rho_{\varphi,\omega}(f)\le1\), then, for every ball \(B\subset\mathbb R^n\) with \(B_\Omega\ne\emptyset\)
		and for a.e. \(x\in B_\Omega\),
		\[
		\psi\!\left(x,\beta\frac1{|B|}\int_{B_\Omega} |f(y)|\,dy\right)
		\le
		C\frac1{|B|}\int_{B_\Omega} \psi(y,|f(y)|)\,dy
		+Ch(x)+C\frac1{|B|}\int_{B_\Omega} h(y)\,dy .
		\]
	\end{proposition}
	
	\begin{proof}
		Set
		\[
		A_B:=\frac1{|B|}\int_{B_\Omega} \psi(y,|f(y)|)\,dy .
		\]
		By Lemma~\ref{lem:normalized-average-bound}, there exists a structural constant
		\(C_0\ge1\), independent of \(f\) and \(B\), such that
		\[
		A_B
		\le
		C_0N_B
		\le C_0L_B .
		\]
		Therefore the hypothesis of Lemma~\ref{lem:local-inverse-jensen} is satisfied with
		\(K=C_0\) and \(u=|f|\). Hence there exist constants \(\beta\in(0,1]\), \(C\ge1\) and
		\[
		h\in L^1_\omega(\Omega)\cap L^\infty(\Omega),
		\qquad h\ge0,
		\]
		depending only on the structural data, such that for a.e. \(x\in B_\Omega\),
		\[
		\psi\!\left(x,\beta\frac1{|B|}\int_{B_\Omega} |f(y)|\,dy\right)
		\le
		C A_B+Ch(x)+C\frac1{|B|}\int_{B_\Omega} h(y)\,dy .
		\]
		Substituting the definition of \(A_B\) gives the desired estimate.
	\end{proof}
	
	\subsection{The pointwise maximal estimate}
	
	\begin{corollary}
		\label{cor:pointwise-normalized-maximal}
		Assume that \((\varphi,\omega)\) satisfies \((\mathcal H_\Omega)\) and that
		\(\omega\in A_{p(\cdot)}(\Omega)\). Then there exist \(\beta\in(0,1]\), \(C\ge1\) and
		\[
		H:=h+\mathcal M_\Omega h,
		\qquad
		h\in L^1_\omega(\Omega)\cap L^\infty(\Omega),
		\qquad h\ge0,
		\]
		such that, whenever \(\rho_{\varphi,\omega}(f)\le1\),
		\[
		\psi(x,\beta \mathcal M_\Omega f(x))
		\le
		C \mathcal M_\Omega(G_f)(x)+CH(x)
		\]
		for a.e. \(x\in\Omega\).
	\end{corollary}
	
	\begin{proof}
		Let \(\beta_0\in(0,1]\), \(C\ge1\) and \(h\) be supplied by
		Proposition~\ref{prop:local-normalized-estimate}. Let \(\mathcal B_{\mathbb Q}\)
		be the countable family of Euclidean balls with rational centers and rational radii.
		Applied to zero extensions, the standard countable-ball reduction gives
		\[
		\mathcal M_\Omega f(x)
		=
		\sup_{\substack{B\in\mathcal B_{\mathbb Q}\\ x\in B}}
		\frac1{|B|}\int_{B_\Omega}|f(y)|\,dy
		\]
		for a.e. \(x\in\Omega\).
		
		For each fixed \(B\in\mathcal B_{\mathbb Q}\), Proposition~\ref{prop:local-normalized-estimate}
		holds for a.e. \(x\in B_\Omega\). Since \(\mathcal B_{\mathbb Q}\) is countable,
		there is a common null set outside which all these estimates hold simultaneously.
		Fix such an \(x\) and set
		\[
		m:=\mathcal M_\Omega f(x),
		\qquad
		R(x):=C\mathcal M_\Omega(G_f)(x)+Ch(x)+C\mathcal M_\Omega h(x).
		\]
		For every \(B\in\mathcal B_{\mathbb Q}\) containing \(x\),
		\[
		\psi\!\left(x,\beta_0\frac1{|B|}\int_{B_\Omega}|f(y)|\,dy\right)
		\le R(x).
		\]
		If \(0<m<\infty\), choose such a ball \(B\) with
		\[
		\frac1{|B|}\int_{B_\Omega}|f(y)|\,dy>\frac m2.
		\]
		By monotonicity of \(\psi(x,\cdot)\),
		\[
		\psi\!\left(x,\frac{\beta_0}{2}m\right)
		\le
		\psi\!\left(x,\beta_0\frac1{|B|}\int_{B_\Omega}|f(y)|\,dy\right)
		\le R(x).
		\]
		The case \(m=0\) is immediate. If \(m=\infty\), the same argument at
		arbitrarily large finite levels, together with
		\(\psi(x,t)\to\infty\) as \(t\to\infty\), shows that \(R(x)=\infty\).
		Thus, with \(\beta:=\beta_0/2\),
		\[
		\psi(x,\beta\mathcal M_\Omega f(x))
		\le
		C\mathcal M_\Omega(G_f)(x)+C\{h(x)+\mathcal M_\Omega h(x)\}.
		\]
		Since \(H=h+\mathcal M_\Omega h\), this is the asserted estimate.
	\end{proof}
	
	\subsection{Modular domination}
	
	\begin{proposition}
		\label{prop:modular-domination}
		Assume that \((\varphi,\omega)\) satisfies \((\mathcal H_\Omega)\) and that
		\((E_\Omega)_\omega\) holds. Then there exist \(\beta\in(0,1]\) and
		\(C\ge1\) such that, if \(\rho_{\varphi,\omega}(f)\le1\), then
		\[
		\rho_{\varphi,\omega}(\beta \mathcal M_\Omega f)\le C.
		\]
	\end{proposition}
	
	\begin{proof}
		The extension hypothesis implies the induced condition \(\omega\in A_{p(\cdot)}(\Omega)\), so Corollary~\ref{cor:pointwise-normalized-maximal} applies. Hence
		\[
		\psi(x,\beta \mathcal M_\Omega f(x))
		\le
		C \mathcal M_\Omega(G_f)(x)+CH(x),
		\qquad H=h+\mathcal M_\Omega h.
		\]
		Since \(h\in L^1_\omega(\Omega)\cap L^\infty(\Omega)\) and \(p^+<\infty\),
		\[
		|h(x)|^{p(x)}
		\le
		\max\{1,\|h\|_\infty^{p^+-1}\}|h(x)|
		\]
		for a.e. \(x\in\Omega\). Consequently,
		\[
		\rho_{p(\cdot),\omega}(h)
		\le
		\max\{1,\|h\|_\infty^{p^+-1}\}
		\|h\|_{L^1_\omega(\Omega)}
		<\infty,
		\]
		and hence \(h\in L^{p(\cdot)}_\omega(\Omega)\). By
		Theorem~\ref{thm:ambient-ve-maximal},
		\(\mathcal M_\Omega h\in L^{p(\cdot)}_\omega(\Omega)\), so
		\(H=h+\mathcal M_\Omega h\in L^{p(\cdot)}_\omega(\Omega)\). Moreover,
		\(\rho_{p(\cdot),\omega}(H)\) is bounded in terms of the fixed auxiliary
		function \(h\) and the structural data.
		
		Using \((a+b)^{p(x)}\le 2^{p^+-1}(a^{p(x)}+b^{p(x)})\), we infer that
		\[
		\begin{aligned}
			\rho_{\varphi,\omega}(\beta \mathcal M_\Omega f)
			&\le
			C\rho_{p(\cdot),\omega}\bigl(\mathcal M_\Omega(G_f)\bigr)
			+C\rho_{p(\cdot),\omega}(H).
		\end{aligned}
		\]
		The first term is bounded by Lemma~\ref{lem:lower-index-maximal-modular}; the second is finite and depends only on the structural
		data. This proves the estimate.
	\end{proof}

	\section{Proofs of the main maximal theorems}
	\label{sec:proof-main-theorems}
	
	In this section we prove the main boundedness and necessity results for the ambient-domain maximal operator and record the whole-space consequence. The proof has two independent parts. The sufficiency
	theory uses the lower-index factorization, the normalized local estimate and the ambient
	weighted variable exponent maximal theorem. The necessity theory uses only the lattice
	property, homogeneity and the ambient ball lower bound.
	
	\subsection{Proof of the lower-index sufficiency theorem}
	
	\begin{proof}[Proof of Theorem~\ref{thm:sufficiency}]
		Set \(p:=p_\varphi\) and \(\psi:=\psi_\varphi\). By assumption,
		\((\varphi,\omega)\) satisfies \((\mathcal H_\Omega)\) and
		\((E_\Omega)_\omega\) holds. The latter implies the induced condition
		\(\omega\in A_{p(\cdot)}(\Omega)\). 
		
		If \(f=0\), there is nothing to prove. Otherwise set
		\[
		\lambda:=2\|f\|_{L^{\varphi(\cdot)}_\omega(\Omega)}.
		\]
		Since \(\lambda>\|f\|_{L^{\varphi(\cdot)}_\omega(\Omega)}\), the definition
		of the Luxemburg functional gives
		\[
		\rho_{\varphi,\omega}(f/\lambda)\le1.
		\]
		By Proposition~\ref{prop:modular-domination} and the homogeneity of
		\(\mathcal M_\Omega\),
		\[
		\rho_{\varphi,\omega}\!\left(
		\frac{\beta\mathcal M_\Omega f}{\lambda}
		\right)
		\le C_0
		\]
		for a structural constant \(C_0\ge1\). Since \(\varphi\in\Phi_w(\Omega)\),
		there is \(L\ge1\) such that
		\[
		\varphi(x,\theta t)\le L\theta\,\varphi(x,t),
		\qquad 0<\theta\le1.
		\]
		Taking \(\theta=(LC_0)^{-1}\), we obtain
		\[
		\rho_{\varphi,\omega}\!\left(
		\theta\frac{\beta\mathcal M_\Omega f}{\lambda}
		\right)
		\le1.
		\]
		Hence
		\[
		\left\|
		\theta\frac{\beta\mathcal M_\Omega f}{\lambda}
		\right\|_{L^{\varphi(\cdot)}_\omega(\Omega)}
		\le1,
		\]
		and therefore
		\[
		\|\mathcal M_\Omega f\|_{L^{\varphi(\cdot)}_\omega(\Omega)}
		\le
		\frac{2LC_0}{\beta}
		\|f\|_{L^{\varphi(\cdot)}_\omega(\Omega)}.
		\]
		This gives the boundedness of \(\mathcal M_\Omega\) on \(L^{\varphi(\cdot)}_\omega(\Omega)\).
	\end{proof}

	\subsection{The associate-space necessity theorem}

Let \(X:=L^{\varphi(\cdot)}_\omega(\Omega)\). The K\"othe associate \(X'\) is always
taken with respect to Lebesgue measure:
\[
\|g\|_{X'}=
\sup_{\|f\|_X\le1}\int_\Omega |f(x)g(x)|\,dx .
\]
This choice is essential because \(\mathcal M_\Omega\) averages with respect to Lebesgue
measure.

The necessity argument only requires the local characteristic-function property
\[
\chi_{B\cap\Omega}\in X
\]
for every Euclidean ball \(B\subset\mathbb R^n\) with \(B\cap\Omega\ne\emptyset\). This
property follows, in particular, from \((A0)^\Omega\) and the definition of a
weight, since for some \(\beta_0>0\),
\[
\int_{B\cap\Omega}\varphi(x,\beta_0)\omega(x)\,dx
\le
\int_{B\cap\Omega}\omega(x)\,dx<\infty .
\]

\begin{proof}[Proof of Theorem~\ref{thm:associate-necessity}]
		Fix a Euclidean ball \(B\subset\mathbb R^n\) and set \(B_\Omega:=B\cap\Omega\).
	By the local characteristic-function assumption, \(\chi_{B_\Omega}\in X\). By
	monotonicity of the Luxemburg functional,
	\(\||f|\chi_{B_\Omega}\|_X\le\|f\|_X\) for all $f \in X$. Hence
	\[
	\|\chi_{B_\Omega}\|_{X'}
	=
	\sup_{\substack{\|f\|_X\le1\\ f\ge0,\ \operatorname{supp} f\subset B_\Omega}}
	\int_{B_\Omega}f(x)\,dx .
	\]
	Let \(f\in X\) be nonnegative and supported in \(B_\Omega\). For every
	\(x\in B_\Omega\), the ball \(B\) is admissible in the definition of
	\(\mathcal M_\Omega\). Hence
	\[
	\mathcal M_\Omega f(x)
	\ge
	\frac1{|B|}\int_{B_\Omega} f(y)\,dy .
	\]
	Therefore
	\[
	\frac1{|B|}\left(\int_{B_\Omega} f\right)\|\chi_{B_\Omega}\|_X
	\le
	\|\mathcal M_\Omega f\|_X
	\le
	\|\mathcal M_\Omega\|_{X\to X}\|f\|_X.
	\]
	Taking the supremum over such \(f\) with \(\|f\|_X\le1\) gives 
	\[
	\frac{\|\chi_{B\cap\Omega}\|_X\|\chi_{B\cap\Omega}\|_{X'}}{|B|}
	\le
	\|\mathcal M_\Omega\|_{X\to X}.
	\]
	Taking the supremum over all Euclidean balls
	\(B\subset\mathbb R^n\) gives
	\(
	[
	\omega
	]_{A_\varphi(\Omega)}<\infty\).
\end{proof}

	\subsection{Reduction from \texorpdfstring{\(A_\varphi\)}{Aphi} to \texorpdfstring{\(A_{p_\varphi(\cdot)}\)}{Apphi(.)}}
	
	The associate condition is necessary in general. The lower-index Muckenhoupt condition is
	obtained from it only under the product reduction introduced in
	Definition~\ref{def:associate-lower-index-reduction}.
	
	\begin{proposition}
		\label{prop:associate-to-lower-index-ap}
		Assume that
		\(
		\omega\in A_\varphi(\Omega)
		\)
		and that \(L^{\varphi(\cdot)}_\omega(\Omega)\) satisfies the associate-to-lower-index
		product reduction. Then
		\(
		\omega\in A_{p_\varphi(\cdot)}(\Omega).
		\)
	\end{proposition}
	
	\begin{proof}
		Assume \(\omega\in A_\varphi(\Omega)\). Then there exists \(C_A\) such that,
		for every Euclidean ball \(B\subset\mathbb R^n\),
		\[
		\|\chi_{B\cap\Omega}\|_{L^{\varphi(\cdot)}_\omega(\Omega)}
		\|\chi_{B\cap\Omega}\|_{(L^{\varphi(\cdot)}_\omega(\Omega))'}
		\le C_A|B|.
		\]
		By the associate-to-lower-index product reduction,
		\[
		\begin{aligned}
			&\|\omega^{1/p_\varphi(\cdot)}\chi_{B\cap\Omega}\|_{L^{p_\varphi(\cdot)}(\Omega)}
			\|\omega^{-1/p_\varphi(\cdot)}\chi_{B\cap\Omega}\|_{L^{p_\varphi'(\cdot)}(\Omega)} \\
			&\qquad\le
			C
			\|\chi_{B\cap\Omega}\|_{L^{\varphi(\cdot)}_\omega(\Omega)}
			\|\chi_{B\cap\Omega}\|_{(L^{\varphi(\cdot)}_\omega(\Omega))'}
			\le CC_A|B|.
		\end{aligned}
		\]
		Taking the supremum over all Euclidean balls \(B\subset\mathbb R^n\) gives
		\(\omega\in A_{p_\varphi(\cdot)}(\Omega)\).
	\end{proof}

	\begin{remark}
		Proposition~\ref{prop:associate-to-lower-index-ap} is purely conditional. The examples below verify its reduction hypothesis only through lower-index power equivalence; the proposition should not be read as asserting that \(A_\varphi\) and \(A_{p_\varphi(\cdot)}\) coincide for general Musielak--Orlicz growth.
	\end{remark}

	\subsection{Whole-space consequence}
	
	\begin{proof}[Proof of Corollary~\ref{cor:whole-space-sufficiency}]
		When \(\Omega=\mathbb R^n\), the ambient-domain maximal operator is the usual
		Hardy--Littlewood maximal operator and the ambient lower-index class coincides with the
		usual variable Muckenhoupt class. The sufficiency direction follows from
		Theorem~\ref{thm:sufficiency}. Conversely, boundedness of \(M\) implies
		\(\omega\in A_\varphi(\mathbb R^n)\) by Theorem~\ref{thm:associate-necessity}; the
		assumed product reduction and Proposition~\ref{prop:associate-to-lower-index-ap} yield
		\(\omega\in A_{p_\varphi(\cdot)}(\mathbb R^n)\).
	\end{proof}

	\section{Density of test functions}
	\label{sec:density}
	
	In this section we prove Theorem~\ref{thm:density-main}. The proof is a direct
	application of the abstract density criterion of Nakai--Tomita--Yabuta
	\cite[Theorem~1.2]{NakaiTomitaYabuta2004}. 
	Therefore, we will verify that our Musielak--Orlicz space satisfies the
	hypotheses of their theorem.
	
	Throughout this section we work on the whole space. Thus the ambient-domain maximal
	operator coincides with the usual Hardy--Littlewood maximal operator and we set
	\[
	X:=L^{\varphi(\cdot)}_\omega(\mathbb R^n).
	\]
	For \(k\in\mathbb N\), define
	\[
	W^kX
	:=
	\left\{
	u\in W^{k,1}_{\mathrm{loc}}(\mathbb R^n):
	D^\alpha u\in X\text{ for every }|\alpha|\le k
	\right\}
	\]
	equipped with the Luxemburg quasi-norm
	\[
	\|u\|_{W^kX}
	:=
	\sum_{|\alpha|\le k}\|D^\alpha u\|_X .
	\]
	Under the assumptions used below, this quasi-norm is equivalent to a Banach norm by
	Proposition~\ref{prop:ambient-bfs}. Thus
	\[
	W^kX
	=
	W^{k,\varphi(\cdot)}_\omega(\mathbb R^n).
	\]
	
	\subsection{The criterion of Nakai--Tomita--Yabuta}
	
	We shall use the following form of the theorem of Nakai--Tomita--Yabuta.
	
	\begin{theorem}
		\label{thm:NTY-density}
		Let \(E\) be a normed or quasi-normed subspace of
		\(L^1_{\mathrm{loc}}(\mathbb R^n)\). For a non-negative integer \(k\), define
		\[
		E_k
		:=
		\left\{
		u\in L^1_{\mathrm{loc}}(\mathbb R^n):
		D^\alpha u\in E\text{ for every }|\alpha|\le k
		\right\},
		\]
		equipped with
		\[
		\|u\|_{E_k}
		:=
		\sum_{|\alpha|\le k}\|D^\alpha u\|_E .
		\]
		Assume that \(E\) has the following properties:
		\begin{enumerate}[label=\textup{(N\arabic*)}]
			\item \(\chi_B\in E\) for every ball \(B\subset\mathbb R^n\);
			\item if \(g\in E\) and \(|f|\le |g|\) a.e., then \(f\in E\);
			\item if \(g\in E\), \(|f_j|\le |g|\) a.e. and \(f_j\to0\) a.e., then
			\(f_j\to0\) in \(E\);
			\item the Hardy--Littlewood maximal operator \(M\) is bounded on \(E\).
		\end{enumerate}
		Then \(C_c^\infty(\mathbb R^n)\) is dense in \(E_k\).
	\end{theorem}

	\subsection{Verification of the criterion}
	
	We now verify the hypotheses of Theorem~\ref{thm:NTY-density} for
	\[
	E=X=L^{\varphi(\cdot)}_\omega(\mathbb R^n).
	\]
	
	We start with the local integrability.
	
	\begin{lemma}
		\label{lem:density-local-integrability}
		Assume that \((\varphi,\omega)\) satisfies
		\((\mathcal H_{\mathbb R^n})\) and that
		\[
		\omega\in A_{p_\varphi(\cdot)}(\mathbb R^n).
		\]
		Then
		\[
		X\subset L^1_{\mathrm{loc}}(\mathbb R^n).
		\]
	\end{lemma}
	
	\begin{proof}
		Let \(f\in X\). Choose \(\lambda_0>0\) such that
		\(\rho_{\varphi,\omega}(f/\lambda_0)<\infty\) and set
		\(g:=f/\lambda_0\). It suffices to prove that
		\(g\in L^1_{\mathrm{loc}}(\mathbb R^n)\). Let \(K\Subset\mathbb R^n\). We split
		\[
		g=g\chi_{\{|g|\le1\}}+g\chi_{\{|g|>1\}}.
		\]
		The first term belongs to \(L^1(K)\). For the second term, split
		\[
		\{|g|>1\}=\{1<|g|<\beta_0^{-1}\}\cup\{|g|\ge\beta_0^{-1}\},
		\]
		where \(\beta_0\) is the constant in \((A0)^{\mathbb R^n}\). On
		\(\{1<|g|<\beta_0^{-1}\}\cap K\), the function \(|g|^{p_\varphi(\cdot)}\) is bounded by a
		constant depending only on \(\beta_0\) and \(p_\varphi^+\) and hence it is integrable with
		respect to \(\omega(x)\,dx\), since \(\omega\in L^1_{\rm loc}\). On
		\(\{|g|\ge\beta_0^{-1}\}\), the normalization \((A0)^{\mathbb R^n}\) and the lower growth
		condition \((aInc)_{p_\varphi(\cdot)}\) imply that there exists a structural constant
		\(C\ge1\) such that
		\[
		|g(x)|^{p_\varphi(x)}
		\le
		C\varphi(x,|g(x)|).
		\]
		Since \(\rho_{\varphi,\omega}(g)<\infty\), it follows that
		\[
		g\chi_{\{|g|>1\}}\in L^{p_\varphi(\cdot)}_\omega(K).
		\]
		By the weighted variable-exponent H\"older inequality,
		\[
		\int_K |g(x)|\chi_{\{|g|>1\}}\,dx
		\le
		C
		\|g\chi_{\{|g|>1\}}\|_{L^{p_\varphi(\cdot)}_\omega(K)}
		\|\omega^{-1/p_\varphi(\cdot)}\chi_K\|_{L^{p_\varphi'(\cdot)}(K)} .
		\]
		The second factor is finite because
		\(\omega\in A_{p_\varphi(\cdot)}(\mathbb R^n)\). Therefore
		\(g\in L^1(K)\). Since \(K\Subset\mathbb R^n\) was arbitrary,
		\(g\in L^1_{\mathrm{loc}}(\mathbb R^n)\) and hence so is
		\(f=\lambda_0g\).
	\end{proof}

Next we verify the requirement on the ball characteristic functions and the lattice property.
	
	\begin{lemma}
		\label{lem:NTY-lattice}
		Assume that \((\varphi,\omega)\) satisfies
		\((\mathcal H_{\mathbb R^n})\). Then \(X\) satisfies \textup{(N1)} and
		\textup{(N2)} of Theorem~\ref{thm:NTY-density}.
	\end{lemma}
	
	\begin{proof}
		Let \(B\subset\mathbb R^n\) be a ball. Since \(\omega\) is locally integrable and
		\((A0)^{\mathbb R^n}\) gives local boundedness of the modular at a fixed small level,
		one can choose \(\lambda>0\) sufficiently large so that
		\[
		\int_B \varphi\left(x,\frac1\lambda\right)\omega(x)\,dx<\infty .
		\]
		Thus \(\chi_B\in X\), proving \textup{(N1)}.
		
		If \(g\in X\) and \(|f|\le |g|\) a.e., then by the monotonicity of
		\(t\mapsto\varphi(x,t)\),
		\[
		\rho_{\varphi,\omega}(f/\lambda)
		\le
		\rho_{\varphi,\omega}(g/\lambda)
		\qquad\text{for every }\lambda>0.
		\]
		Hence \(f\in X\). This proves \textup{(N2)}.
	\end{proof}

Next we deal with the dominated convergence in \(X\).
	
	\begin{lemma}
		\label{lem:NTY-dominated-convergence}
		Assume that \((\varphi,\omega)\) satisfies
		\((\mathcal H_{\mathbb R^n})\) and that \(\psi_\varphi\) satisfies
		\((aDec)_{q_\psi}\) for some \(q_\psi<\infty\). Let \(g\in X\),
		\[
		|f_j|\le |g|\quad\text{a.e.}
		\]
		and
		\[
		f_j\to0\quad\text{a.e. in }\mathbb R^n.
		\]
		Then
		\[
		\|f_j\|_X\to0 .
		\]
		Consequently, \(X\) satisfies \textup{(N3)} of
		Theorem~\ref{thm:NTY-density}.
	\end{lemma}
	
	\begin{proof}
		Since \(g\in X\), there exists \(\lambda_0>0\) such that
		\[
		\rho_{\varphi,\omega}(g/\lambda_0)<\infty .
		\]
		The finite upper-growth assumption \((aDec)_{q_\psi}\), together with
		\(p_\varphi^+<\infty\), gives a finite \(\Delta_2\)-type control for \(\varphi\).
		Therefore,
		\[
		\rho_{\varphi,\omega}(g/\lambda)<\infty
		\qquad\text{for every }\lambda>0.
		\]
		Fix \(\lambda>0\). Since
		\[
		|f_j|\le |g|
		\quad\text{and}\quad
		f_j\to0\quad\text{a.e.},
		\]
		we have
		\[
		\varphi\left(x,\frac{|f_j(x)|}{\lambda}\right)
		\le
		\varphi\left(x,\frac{|g(x)|}{\lambda}\right)
		\quad\text{a.e.}
		\]
		and the right-hand side is integrable with respect to \(\omega(x)\,dx\). Hence the
		dominated convergence theorem gives
		\[
		\rho_{\varphi,\omega}(f_j/\lambda)
		=
		\int_{\mathbb R^n}
		\varphi\left(x,\frac{|f_j(x)|}{\lambda}\right)\omega(x)\,dx
		\to0 .
		\]
		Since this holds for every \(\lambda>0\), the Luxemburg quasi-norm satisfies
		\[
		\|f_j\|_X\to0 
		\]
		as required.
	\end{proof}

The next property we address is the boundedness of the maximal function.
	
	\begin{lemma}
		\label{lem:NTY-maximal-boundedness}
		Assume that \((\varphi,\omega)\) satisfies
		\((\mathcal H_{\mathbb R^n})\), that
		\[
		\omega\in A_{p_\varphi(\cdot)}(\mathbb R^n),
		\]
		and that \(\psi_\varphi\) satisfies \((aDec)_{q_\psi}\) for some
		\(q_\psi<\infty\). Then
		\[
		M:X\to X
		\]
		is bounded. Hence \(X\) satisfies \textup{(N4)} of
		Theorem~\ref{thm:NTY-density}.
	\end{lemma}
	
	\begin{proof}
		By \((\mathcal H_{\mathbb R^n})\), the lower-index exponent
		\(p_\varphi\) belongs to \(P^{\log}(\mathbb R^n)\). Since
		\[
		\omega\in A_{p_\varphi(\cdot)}(\mathbb R^n),
		\]
		the whole-space weighted variable-exponent maximal theorem applies. Hence the
		lower-index sufficiency theorem, Theorem~\ref{thm:sufficiency} with
		\(\Omega=\mathbb R^n\), gives
		\[
		\|Mf\|_{L^{\varphi(\cdot)}_\omega(\mathbb R^n)}
		\le
		C
		\|f\|_{L^{\varphi(\cdot)}_\omega(\mathbb R^n)} 
		\]
		as required.
	\end{proof}

Putting the above verification together, we arrive at the following result.
	
	\begin{proposition}
		\label{prop:NTY-verification}
		Assume that \((\varphi,\omega)\) satisfies
		\((\mathcal H_{\mathbb R^n})\), that
		\[
		\omega\in A_{p_\varphi(\cdot)}(\mathbb R^n),
		\]
		and that \(\psi_\varphi\) satisfies \((aDec)_{q_\psi}\) for some
		\(q_\psi<\infty\). Then
		\[
		X=L^{\varphi(\cdot)}_\omega(\mathbb R^n)
		\]
		satisfies all assumptions of Theorem~\ref{thm:NTY-density}.
	\end{proposition}
	
	\begin{proof}
		By Lemma~\ref{lem:density-local-integrability}, \(X\subset
		L^1_{\mathrm{loc}}(\mathbb R^n)\). Lemma~\ref{lem:NTY-lattice} gives
		\textup{(N1)} and \textup{(N2)}. Lemma~\ref{lem:NTY-dominated-convergence}
		gives \textup{(N3)}. Lemma~\ref{lem:NTY-maximal-boundedness} gives
		\textup{(N4)}. Hence all hypotheses of Theorem~\ref{thm:NTY-density} are satisfied.
	\end{proof}
	
	\subsection{Proof of the density theorem}
	
	We now prove Theorem~\ref{thm:density-main}.
	
	\begin{proof}[Proof of Theorem~\ref{thm:density-main}]
		Set
		\[
		X=L^{\varphi(\cdot)}_\omega(\mathbb R^n).
		\]
		By Proposition~\ref{prop:NTY-verification}, the space \(X\) satisfies the assumptions of
		the Nakai--Tomita--Yabuta density criterion, Theorem~\ref{thm:NTY-density}. Therefore
		\[
		C_c^\infty(\mathbb R^n)
		\quad\text{is dense in}\quad
		E_k,
		\]
		where
		\[
		E_k
		=
		\left\{
		u\in L^1_{\mathrm{loc}}(\mathbb R^n):
		D^\alpha u\in X\text{ for every }|\alpha|\le k
		\right\}
		\]
		and
		\[
		\|u\|_{E_k}
		=
		\sum_{|\alpha|\le k}\|D^\alpha u\|_X .
		\]
		By the definition of \(X\), this space is exactly
		\[
		E_k
		=
		W^{k,\varphi(\cdot)}_\omega(\mathbb R^n)
		\]
		with equality of norms. Hence
		\[
		\overline{C_c^\infty(\mathbb R^n)}
		^{\;W^{k,\varphi(\cdot)}_\omega(\mathbb R^n)}
		=
		W^{k,\varphi(\cdot)}_\omega(\mathbb R^n).
		\]
		This proves Theorem~\ref{thm:density-main}.
	\end{proof}

	\section{Examples and verification of the structural hypotheses}
	\label{sec:examples}
	
	In this section we verify the hypotheses of the main theorems in concrete examples. Throughout
	this section the notation \(A_p(\Omega)\), \(A_{p(\cdot)}(\Omega)\) and
	\(A_\varphi(\Omega)\) refers to the induced domain classes associated with the ambient
	maximal operator \(\mathcal M_\Omega\): all Euclidean balls
	\(B\subset\mathbb R^n\) are tested through \(B\cap\Omega\), with the normalization by
	\(|B|\). We write \(X_\varphi:=L^{\varphi(\cdot)}_\omega(\Omega)\) and
	\(B_\Omega:=B\cap\Omega\).
	
	The lower-index class \(A_{p_\varphi(\cdot)}(\Omega)\) is the natural domain testing
	condition used in the local estimates and in the associate-to-lower-index necessity mechanism.
	By \((H1)\), the lower-index exponent \(p_\varphi\) and its whole-space log-H\"older
	extension \(\widetilde p_\varphi\) are fixed. The condition \((E_\Omega)_\omega\) is then a
	condition on the admissibility of the weight: it requires a whole-space weight
	\[
	\widetilde\omega\in A_{\widetilde p_\varphi(\cdot)}(\mathbb R^n)
	\]
	whose restriction to \(\Omega\) is \(\omega\). This permits the identity
	\[
	\mathcal M_\Omega f=M(f\chi_\Omega)|_\Omega
	\]
	to be combined with the whole-space weighted variable exponent maximal theorem. Thus, in each
	domain-level sufficiency statement below, \((E_\Omega)_\omega\) is part of the
	hypotheses. When \(\Omega=\mathbb R^n\), it reduces to the usual whole-space
	lower-index condition \(\omega\in A_{p_\varphi(\cdot)}(\mathbb R^n)\).
	In fixed-index cases on a proper domain, it means that \(\omega\) is the restriction
	of a whole-space \(A_p(\mathbb R^n)\)-weight and it can also be
	verified by known extension results for Muckenhoupt weights, such as Wolff's extension theorem
	in the induced setting. In genuinely variable-index examples, we keep \((E_\Omega)_\omega\)
	as an explicit assumption. This is not merely technical bookkeeping: on a proper, non-smooth domain an induced \(A_{p(\cdot)}(\Omega)\)-condition need not come with a known whole-space extension, and no general variable-exponent analogue of the fixed-exponent extension criterion is used here.

	Under its local characteristic-function hypothesis, the converse theorem gives the ambient
	associate condition \(A_\varphi(\Omega)\). Thus \(A_\varphi(\Omega)\) is the natural
	necessary class, while the lower-index class is the
	robust sufficient class once the structural assumptions and \((E_\Omega)_\omega\) are available.
	For density, we also require \((aDec)_{q_\psi}\). For the double-phase background and
	related nonstandard growth phenomena we refer to Zhikov \cite{Zhikov1986,Zhikov1995} and to
	Colombo--Mingione and Baroni--Colombo--Mingione
	\cite{ColomboMingione2015,BaroniColomboMingione2018}. 
	In examples where order-continuity or dominated convergence in
	\(L^{\varphi(\cdot)}_\omega\) is needed, it follows from
	Proposition~\ref{prop:ambient-bfs}. When a complementary-function
	representation of the K\"othe associate is used, we explicitly assume an
	equivalent convex representative and the standard Musielak--Orlicz duality
	theorem.
	
	\subsection{Pure power growth}
	
	Let \(\varphi(x,t)=t^p\), \(1<p<\infty\). Then \(p_\varphi=p\),
	\(\psi_\varphi(x,t)=t\), \((aInc)_p\), \((A0)^\Omega\), \((A1)^\Omega_{\omega,p}\),
	\((A2)^\Omega_\omega\) and \((aDec)_1\) hold exactly. Moreover,
	\(X_\varphi=L^p_\omega(\Omega)\) and
	\(X_\varphi'=L^{p'}_{\omega^{-p'/p}}(\Omega)\). Hence
	\[
	\frac{\|\chi_{B_\Omega}\|_{X_\varphi}\|\chi_{B_\Omega}\|_{X_\varphi'}}{|B|}
	=
	\left(\frac1{|B|}\int_{B_\Omega}\omega\,dx\right)^{1/p}
	\left(\frac1{|B|}\int_{B_\Omega}\omega^{-\frac1{p-1}}\,dx\right)^{1/p'}.
	\]
	Therefore
	\[
	A_\varphi(\Omega)=A_{p_\varphi(\cdot)}(\Omega)=A_p(\Omega),
	\]
	and the associate-to-lower-index product reduction holds with equality. The sufficiency
	statement is obtained under \((E_\Omega)_\omega\). This holds, for instance, if
	\(\omega=\widetilde\omega|_\Omega\) with
	\(\widetilde\omega\in A_p(\mathbb R^n)\). It may also be supplied by Wolff's extension
	theorem: in the induced setting, the strengthened condition
	\(\omega^{1+\varepsilon}\in A_p(\Omega)\) for some \(\varepsilon>0\) yields an
	extension \(\widetilde\omega\in A_p(\mathbb R^n)\); see
	\cite{KurkiMudarra2022}.
	
	\subsection{Bounded spatially modulated powers}
	
	Let
	\[
	\varphi(x,t)=\theta(x)t^p,
	\qquad
	1<p<\infty,
	\]
	where \(0<\theta_-\le\theta(x)\le\theta_+<\infty\). Then \(p_\varphi=p\) and
	\(\psi_\varphi(x,t)=\theta(x)^{1/p}t\). The condition \((aInc)_p\) holds with constant
	one. The normalization \((A0)^\Omega\) follows from
	\(\psi_\varphi^{-1}(x,1)=\theta(x)^{-1/p}\). The inverse conditions \((A1)\) and \((A2)\)
	follow from the two-sided bounds on \(\theta\), with \(h_s\equiv0\). The condition
	\((aDec)_1\) is immediate. Since the space is equivalent to \(L^p_\omega(\Omega)\),
	\[
	A_\varphi(\Omega)=A_p(\Omega)=A_{p_\varphi(\cdot)}(\Omega).
	\]
	The boundedness conclusion is obtained under the same fixed-index extension hypothesis
	\((E_\Omega)_\omega\) as in the pure power case.
	
	\subsection{Bi-Lipschitz perturbations of powers}
	
	Let
	\[
	\varphi(x,t)=\theta(x)\Psi(t)^p,
	\]
	where \(0<\theta_-\le\theta\le\theta_+<\infty\) and \(\Psi\) is an increasing
	scalar weak \(\Phi\)-function satisfying \(c_1t\le\Psi(t)\le c_2t\). Then
	\(\varphi(x,t)\simeq\theta(x)t^p\) uniformly and the model is equivalent to the bounded
	modulated power case. Hence \(p_\varphi=p\), the structural conditions hold with
	constants depending on \(\theta_\pm,c_1,c_2\) and
	\[
	A_\varphi(\Omega)=A_p(\Omega)=A_{p_\varphi(\cdot)}(\Omega).
	\]
	The boundedness conclusion is obtained under the same fixed-index extension hypothesis
	\((E_\Omega)_\omega\) as in the pure power case.
	
	\subsection{Pure variable exponent growth}
	
	Let \(\varphi(x,t)=t^{p(x)}\), where \(p\in P^{\log}(\Omega)\) and
	\(1<p^-\le p^+<\infty\). Assume also that a whole-space log-H\"older extension
	\(\widetilde p\) of \(p\) has been fixed as in \((H1)\). Then \(p_\varphi=p\),
	\(\psi_\varphi(x,t)=t\) and all structural conditions are immediate. Moreover
	\(X_\varphi=L^{p(\cdot)}_\omega(\Omega)\) and its K\"othe associate is represented by
	\(L^{p'(\cdot)}_{\omega^{-p'(\cdot)/p(\cdot)}}(\Omega)\). Thus
	\[
	\omega\in A_\varphi(\Omega)
	\quad\Longleftrightarrow\quad
	\omega\in A_{p(\cdot)}(\Omega).
	\]
	The product reduction is the standard variable-exponent associate reduction. The sufficiency
	statement requires \((E_\Omega)_\omega\). Here \((H1)\) fixes a whole-space log-H\"older
	extension \(\widetilde p\) of \(p\) and \((E_\Omega)_\omega\) requires a weight
	\[
	\widetilde\omega\in A_{\widetilde p(\cdot)}(\mathbb R^n)
	\]
	with \(\widetilde\omega=\omega\) a.e. on \(\Omega\). Under this extension hypothesis,
	Theorem~\ref{thm:sufficiency} gives the domain boundedness result. Conversely,
	Theorem~\ref{thm:associate-necessity} and the standard variable-exponent associate reduction
	identify \(A_{p(\cdot)}(\Omega)\) as the corresponding lower-index necessary condition.
	When \(\Omega=\mathbb R^n\), Corollary~\ref{cor:whole-space-sufficiency} recovers the
	whole-space weighted variable-exponent characterization in \cite{CruzUribeDieningHasto2011}.
	
	\subsection{Scalar Orlicz growth}
	
	Let \(\varphi(x,t)=\Phi(t)\), where \(\Phi\) is a Young function. Assume that the scalar
	function \(\varphi\) satisfies the specialization of our standing hypotheses
	\((\mathcal H_\Omega)\). Thus
	\[
	p_\varphi=i(\Phi),
	\qquad
	1<i(\Phi)<\infty,
	\]
	\(\Phi\) satisfies \((aInc)_{i(\Phi)}\) and the normalized function
	\[
	\psi_\varphi(t)=\Phi(t)^{1/i(\Phi)}
	\]
	satisfies the corresponding scalar normalization assumptions. Since there is no
	\(x\)-dependence, the local comparison conditions \((A1)\) and \((A2)\) are automatic,
	with \(h_s\equiv0\), after normalization. For the density theorem one also assumes
	\((aDec)_{q_\psi}\) for \(\psi_\varphi\).
	
	Let \(\widetilde\Phi\) denote the complementary Young function of
	\(\Phi\).
	The associate condition becomes
	\[
	\sup_{B\subset\mathbb R^n}
	\frac{
		\|\chi_{B_\Omega}\|_{L^\Phi_\omega(\Omega)}
		\|\omega^{-1}\chi_{B_\Omega}\|_{L^{\widetilde\Phi}_\omega(\Omega)}
	}{|B|}<\infty .
	\]
	We denote this associate condition by \(\omega\in A_\Phi(\Omega)\). The lower-index
	condition
	\[
	\omega\in A_{i(\Phi)}(\Omega)
	\]
	together with the fixed-index extension hypothesis \((E_\Omega)_\omega\), gives the
	sufficiency conclusion by Theorem~\ref{thm:sufficiency}, whereas boundedness of
	\(\mathcal M_\Omega\) gives
	\[
	\omega\in A_\Phi(\Omega)
	\]
	by Theorem~\ref{thm:associate-necessity}. Thus the results of this paper give
	\[
	\omega\in A_{i(\Phi)}(\Omega)\ \text{ and }\ (E_\Omega)_\omega
	\quad\Longrightarrow\quad
	\mathcal M_\Omega\ \text{is bounded}
	\quad\Longrightarrow\quad
	\omega\in A_\Phi(\Omega).
	\]
	We do not claim that \(A_\Phi(\Omega)=A_{i(\Phi)}(\Omega)\). The classical references
	\cite{BloomKerman1994,KokilashviliKrbec1991} concern stronger modular inequalities and do
	not establish the necessity of \(A_{i(\Phi)}\) for Luxemburg-norm boundedness. As in the
	pure-power case, the extension hypothesis holds when \(\omega\) is the restriction of a
	whole-space \(A_{i(\Phi)}(\mathbb R^n)\)-weight and may also be verified by Wolff's
	extension theorem under the strengthened induced condition
	\(\omega^{1+\varepsilon}\in A_{i(\Phi)}(\Omega)\).

	\subsection{Nondegenerate double-phase growth}
	
	Let
	\[
	\varphi(x,t)=t^p+a(x)t^q,
	\qquad
	1<p<q<\infty,
	\]
	where \(a:\Omega\to[0,\infty)\) is measurable and \(0<a_-\le a(x)\le a_+<\infty\). Then
	\(\varphi(x,t)\simeq t^p+t^q\) uniformly, so the model is equivalent to the scalar Orlicz
	function \(\Phi(t)=t^p+t^q\). The lower index is \(p_\varphi=p\). The condition
	\((aInc)_p\) follows since \(\varphi(x,t)/t^p=1+a(x)t^{q-p}\) is increasing. The normalized
	function
	\[
	\psi_\varphi(x,t)=t(1+a(x)t^{q-p})^{1/p}
	\]
	is uniformly equivalent to \((t^p+t^q)^{1/p}\), so \((A0)\), \((A1)\), \((A2)\) and
	\((aDec)_{q/p}\) hold. Hence \(\omega\in A_p(\Omega)\), together with
	\((E_\Omega)_\omega\), gives boundedness of \(\mathcal M_\Omega\) by
	Theorem~\ref{thm:sufficiency}. The converse gives the associate condition
	\(A_\varphi(\Omega)\). Since this model is uniformly equivalent to the scalar Orlicz
	function \(\Phi(t)=t^p+t^q\), the corresponding weighted spaces and their K\"othe
	associates have equivalent norms, so \(A_\varphi(\Omega)\) agrees, up to comparable
	constants, with the scalar associate condition \(A_\Phi(\Omega)\). Thus
	\[
	\omega\in A_p(\Omega)\ \text{ and }\ (E_\Omega)_\omega
	\quad\Longrightarrow\quad
	\mathcal M_\Omega\ \text{is bounded}
	\quad\Longrightarrow\quad
	\omega\in A_\Phi(\Omega).
	\]
	No identification \(A_\Phi(\Omega)=A_p(\Omega)\) is asserted. Consequently, this genuinely two-phase model does not fall under the conditional characterization of Corollary~\ref{cor:whole-space-sufficiency}; only the sufficiency/necessity chain displayed above is obtained. Establishing the product reduction for this model remains open. The domain boundedness conclusion still uses the extension hypothesis \((E_\Omega)_\omega\).

	\subsection{Fully explicit saturated families with product reduction}

	We finish the verification section with saturated multiphase models for which every
	assumption used in the maximal theorem, the finite upper-growth condition used in the
	density theorem, and the associate-to-lower-index product reduction can all be checked
	directly.  Saturation is useful here because it preserves genuine dependence on both
	\(x\) and \(t\), and may allow degenerate higher-phase coefficients, while keeping the
	growth uniformly equivalent to its lower-index power.
	
	The core idea in this subsection is the power-equivalent transfer of the product reduction, as presented next.

	\begin{lemma}
		\label{lem:power-equivalent-product-reduction}
		Let \(p\in P(\Omega)\), let \(\omega\) be a weight, and suppose that
		\(p_\varphi=p\) and
		\[
		c\,t^{p(x)}
		\le
		\varphi(x,t)
		\le
		C\,t^{p(x)}
		\]
		for a.e. \(x\in\Omega\), every \(t\ge0\), and fixed constants
		\(0<c\le C<\infty\).  Then
		\[
		L^{\varphi(\cdot)}_\omega(\Omega)
		=
		L^{p(\cdot)}_\omega(\Omega)
		\]
		with equivalent norms, and their K\"othe associate norms are equivalent.  Consequently,
		for every measurable set \(E\subset\Omega\),
		\[
		\begin{aligned}
		&\|\chi_E\|_{L^{\varphi(\cdot)}_\omega(\Omega)}
		\|\chi_E\|_{(L^{\varphi(\cdot)}_\omega(\Omega))'}
		\simeq
		\|\omega^{1/p(\cdot)}\chi_E\|_{L^{p(\cdot)}(\Omega)}
		\|\omega^{-1/p(\cdot)}\chi_E\|_{L^{p'(\cdot)}(\Omega)}.
		\end{aligned}
		\]
		In particular, the associate-to-lower-index product reduction holds, in fact with
		two-sided comparability of the two products.
	\end{lemma}

	\begin{proof}
		The pointwise comparison of the modulars gives equivalence of the Luxemburg norms.
		Equivalent lattice quasi-norms induce equivalent K\"othe-associate norms.  The displayed
		comparison then follows from the standard associate formula for
		\(L^{p(\cdot)}_\omega(\Omega)\).
	\end{proof}

Based on Lemma \ref{lem:power-equivalent-product-reduction}, we now construct several examples which satisfy, in particular, the reduction relation.

	\begin{proposition}
		\label{prop:saturated-multiphase-family}
		Let \(p\in P^{\log}(\Omega)\), fix a whole-space log-H\"older extension
		\(\widetilde p\in P^{\log}(\mathbb R^n)\), and let
		\(m\in\mathbb N\).  Suppose that, for \(j=1,\dots,m\),
		\[
		\sigma_j:\Omega\to(0,\infty),
		\qquad
		0\le a_j(x)\le A_j<\infty
		\]
		are measurable.  Define
		\[
		\varphi(x,t)
		:=
		t^{p(x)}
		\left[
		1+\sum_{j=1}^{m}a_j(x)
		\frac{t^{\sigma_j(x)}}{1+t^{\sigma_j(x)}}
		\right],
		\qquad t\ge0.
		\]
		Then \(\varphi\in\Phi_w(\Omega)\),
		\[
		p_\varphi(x)=p(x)
		\qquad\text{for a.e. }x\in\Omega,
		\]
		and, for every weight \(\omega\), the pair \((\varphi,\omega)\) satisfies
		\((\mathcal H_\Omega)\).  Moreover, \(\psi_\varphi\) satisfies
		\((aDec)_1\), and the associate-to-lower-index product reduction holds.
	\end{proposition}

	\begin{proof}
		Set
		\[
		H_x(t)
		:=
		1+\sum_{j=1}^{m}a_j(x)
		\frac{t^{\sigma_j(x)}}{1+t^{\sigma_j(x)}},
		\qquad
		K:=1+\sum_{j=1}^{m}A_j.
		\]
		For a.e. \(x\), the function \(H_x\) is increasing,
		\[
		1\le H_x(t)\le K,
		\qquad
		\lim_{t\to0+}H_x(t)=1.
		\]
		Since \(p^->1\), the function
		\(t\mapsto \varphi(x,t)/t=t^{p(x)-1}H_x(t)\) is increasing.  Thus
		\(\varphi\in\Phi_w(\Omega)\).

		For \(0<\lambda<1\),
		\[
		\frac{\varphi(x,\lambda t)}{\varphi(x,t)}
		=
		\lambda^{p(x)}\frac{H_x(\lambda t)}{H_x(t)}
		\le
		\lambda^{p(x)}.
		\]
		Letting \(t\to0+\) shows that the supremum in the definition of
		\(g_\varphi(x,\lambda)\) is exactly \(\lambda^{p(x)}\).  Hence
		\(p_\varphi(x)=p(x)\).  Also,
		\[
		\frac{\varphi(x,t)}{t^{p(x)}}=H_x(t)
		\]
		is increasing, so \((aInc)_{p(\cdot)}\) holds with constant one.

		The normalized function is
		\[
		\psi_\varphi(x,t)
		=
		t H_x(t)^{1/p(x)}.
		\]
		Writing \(C_0:=K^{1/p^-}\), we have the uniform comparison
		\begin{equation}
		\label{eq:saturated-psi-comparison}
		t\le\psi_\varphi(x,t)\le C_0t.
		\end{equation}
		In particular, \(\psi_\varphi\in\Phi_w(\Omega)\).  Its inverse satisfies
		\[
		C_0^{-1}s
		\le
		\psi_\varphi^{-1}(x,s)
		\le s,
		\qquad s\ge0.
		\]
		It follows directly that \((A0)^\Omega\) holds with
		\(\beta_0=C_0^{-1}\), that \((A1)^\Omega_{\omega,p}\) holds with
		\(\beta_1=C_0^{-1}\) for every weight \(\omega\), and that
		\((A2)^\Omega_\omega\) holds with \(h_s\equiv0\) and
		\(\beta_s=C_0^{-1}\) for every \(s>0\).  Hence all parts of
		\((\mathcal H_\Omega)\) are verified.

		For \(\lambda\ge1\), using \eqref{eq:saturated-psi-comparison},
		\[
		\psi_\varphi(x,\lambda t)
		\le C_0\lambda t
		\le C_0\lambda\psi_\varphi(x,t),
		\]
		so \(\psi_\varphi\) satisfies \((aDec)_1\).  Finally,
		\[
		t^{p(x)}\le\varphi(x,t)\le Kt^{p(x)}
		\]
		and Lemma~\ref{lem:power-equivalent-product-reduction} gives the product reduction.
	\end{proof}

	\begin{example}
		Let \(\Omega\subset\mathbb R^n\) be open, fix \(x_0\in\Omega\), \(1<p<q<\infty\), and \(\gamma>0\).  Let
		\[
		a(x):=
		\frac{|x-x_0|^\gamma}{1+|x-x_0|^\gamma},
		\qquad
		\varphi(x,t)
		:=
		t^p+a(x)\frac{t^q}{1+t^{q-p}}.
		\]
		Then
		\[
		\varphi(x,t)
		=
		t^p\left(1+a(x)\frac{t^{q-p}}{1+t^{q-p}}\right),
		\]
		so Proposition~\ref{prop:saturated-multiphase-family} applies with
		\(m=1\), \(p(x)\equiv p\), \(\sigma_1=q-p\), and \(A_1=1\).
		Consequently,
		\[
		p_\varphi=p,
		\qquad
		t^p\le\varphi(x,t)\le2t^p,
		\]
		\((\mathcal H_\Omega)\) holds, \(\psi_\varphi\) satisfies
		\((aDec)_1\), and the product reduction holds.  Notice that
		\(a(x_0)=0\), so the higher phase is genuinely degenerate, while for small \(t\)
		one has the double-phase expansion
		\[
		\varphi(x,t)=t^p+a(x)t^q+o(t^q).
		\]

		Choose
		\[
		-n<\alpha<n(p-1),
		\qquad
		\omega(x):=|x-x_0|^\alpha\big|_\Omega.
		\]
		The whole-space extension
		\(\widetilde\omega(x)=|x-x_0|^\alpha\) belongs to
		\(A_p(\mathbb R^n)\), so \((E_\Omega)_\omega\) holds.  Therefore
		Theorem~\ref{thm:sufficiency} applies on every open set \(\Omega\).  On
		\(\mathbb R^n\), the finite upper-growth condition and
		\(\widetilde\omega\in A_p\) also verify all hypotheses of
		Theorem~\ref{thm:density-main}.  In addition,
		\[
		A_\varphi(\Omega)=A_p(\Omega)
		\]
		with comparable defining constants.
	\end{example}

	\begin{example}
		Let \(\Omega\subset\mathbb R^n\) be open and fix \(x_0\in\Omega\), \(p_\infty>1\), \(\delta>0\), \(\sigma>0\), and \(\gamma>0\).  On
		\(\mathbb R^n\), define
		\[
		p(x):=p_\infty+\frac{\delta}{\log(e+|x-x_0|)},
		\qquad
		q(x):=p(x)+\sigma,
		\]
		\[
		a(x):=\frac{|x-x_0|^\gamma}{1+|x-x_0|^\gamma},
		\qquad
		\widetilde\omega(x):=2+\cos x_1.
		\]
		For an arbitrary open set \(\Omega\subset\mathbb R^n\), restrict these data to
		\(\Omega\) and set
		\[
		\varphi(x,t)
		:=
		t^{p(x)}+a(x)\frac{t^{q(x)}}{1+t^\sigma}
		=
		t^{p(x)}\left(1+a(x)\frac{t^\sigma}{1+t^\sigma}\right).
		\]
		The exponent \(p\) is locally Lipschitz, satisfies
		\( |p(x)-p_\infty|\le\delta/\log(e+|x-x_0|)\), and is therefore globally
		log-H\"older continuous.  Moreover,
		\[
		p_\infty\le p(x)\le p_\infty+\delta,
		\]
		and Proposition~\ref{prop:saturated-multiphase-family} gives
		\[
		p_\varphi(x)=p(x),
		\qquad
		t^{p(x)}\le\varphi(x,t)\le2t^{p(x)}.
		\]
		Thus \((\mathcal H_\Omega)\), \((aDec)_1\) for \(\psi_\varphi\), and the
		product reduction all hold.  Since
		\[
		1\le\widetilde\omega(x)\le3,
		\]
		the weight is equivalent to the constant weight and belongs to
		\(A_{p(\cdot)}(\mathbb R^n)\).  Hence \((E_\Omega)_\omega\) holds for
		\(\omega=\widetilde\omega|_\Omega\).  The maximal theorem therefore applies on
		\(\Omega\), and on \(\mathbb R^n\) the density theorem applies as well.  Finally,
		\[
		A_\varphi(\Omega)=A_{p(\cdot)}(\Omega)
		\]
		with comparable defining constants.  This example has a nonconstant lower index,
		a degenerate higher-phase coefficient, genuine \(x\)- and \(t\)-dependence, and a
		nonconstant weight.
	\end{example}

	\begin{remark}[Saturated versus unsaturated phases]
		The saturated models above satisfy the product reduction because they remain uniformly
		power-equivalent.  In contrast, the unsaturated model
		\(
		t^p+a(x)t^q
		\)
		is not uniformly equivalent to its lower-index power when the higher phase is active.
		For this model, whenever \((\varphi,\omega)\) satisfies \((\mathcal H_{\mathbb R^n})\), the present paper gives the safe whole-space implication chain. On \(\mathbb R^n\), we write \(M:=\mathcal M_{\mathbb R^n}\). Thus,
		\[
		\omega\in A_{p_\varphi(\cdot)}(\mathbb R^n)
		\Longrightarrow
		M\text{ is bounded on }L^{\varphi(\cdot)}_\omega(\mathbb R^n)
		\Longrightarrow
		\omega\in A_\varphi(\mathbb R^n),
		\]
		but does not assert the associate-to-lower-index product reduction.
	\end{remark}

	\subsection{Summary of the reduction-verified saturated examples}

	For the two explicit families above, every structural, extension, upper-growth, and reduction
	condition used in the paper is verified.  More precisely,
	\[
	\renewcommand{\arraystretch}{1.25}
	\begin{array}{c|c|c|c}
		\text{Growth} & p_\varphi & \text{Extension weight} & \text{Reduction}
		\\ \hline
		t^p+a(x)\dfrac{t^q}{1+t^{q-p}}
		& p
		& |x-x_0|^\alpha,\ -n<\alpha<n(p-1)
		& A_\varphi=A_p
		\\[2mm]
		t^{p(x)}+a(x)\dfrac{t^{p(x)+\sigma}}{1+t^\sigma}
		& p(x)
		& 2+\cos x_1
		& A_\varphi=A_{p(\cdot)}
	\end{array}
	\]
	Here the equalities of classes are understood with comparable defining constants.  In both
	rows, \((\mathcal H_\Omega)\), \((E_\Omega)_\omega\), and
	\((aDec)_1\) for \(\psi_\varphi\) hold; hence both the maximal theorem and, on the whole
	space, the density theorem apply.

	\section*{Acknowledgments}
	The author is grateful to Vertti Hietanen for valuable comments on an earlier version of the manuscript. These comments led to corrections and clarifications in the scalar Orlicz and double-phase discussions.

\end{document}